\documentclass[12pt,a4paper]{article}

\usepackage[utf8]{inputenc}
\usepackage{amssymb, amsthm, amsmath}
\usepackage[english]{babel}
\usepackage[T1]{fontenc}
\usepackage{graphicx}
\usepackage{mathtools}
\usepackage[font=small,labelfont={normal, bf}]{caption}
\usepackage{subcaption}
\usepackage{xcolor}

\usepackage[cm]{fullpage}

\usepackage{autonum}

\makeatletter
\newcommand{\@giventhatstar}[2]{#1\;\middle|\;#2)}
\newcommand{\@giventhatnostar}[3][]{#1#2\;#1|\;#3#1}
\newcommand{\giventhat}{\@ifstar\@giventhatstar\@giventhatnostar}
\makeatother

\newtheorem{theorem}{Theorem}
\newtheorem{proposition}{Proposition}
\theoremstyle{definition}

\newtheorem{example}{Example}

\title{Asymmetric L\'evy walks driven by convex combination of fractional material derivatives}
\author{
Łukasz Płociniczak\thanks{Faculty of Pure and Applied Mathematics, Wroclaw University of Science and Technology, Wyb. Wyspia\'nskiego 27, 50-370 Wroc{\l}aw, Poland}, \and 
Marek A. Teuerle$^*$, \and
Hubert Woszczek$^{*,}$\thanks{\underline{Corresponding author:} \texttt{hubert.woszczek@pwr.edu.pl}}
}
\date{}

\begin{document}

\maketitle

\begin{abstract}
We analyze a class of linear partial differential equations that arise as deterministic descriptions of the scaling limits of L\'evy walks, in which transport is driven by a convex combination of fractional material derivatives and a source term. Using techniques of Fourier-Laplace transforms, we first prove the existence of mild solutions for continuous initial data. Using a recently obtained pointwise representation of the fractional material derivative, we then identify a necessary and sufficient condition on the source term that guaranties the solution to remain a probability density for all times (non-negativity and unit mass). Motivated by the need to preserve these probabilistic properties in computations, we construct a finite-volume discretization that is probability conservative by construction. We establish discrete stability and a convergence result for the continuous weak solution as space and time steps tend to zero. Extensive numerical experiments validate the scheme: total mass is conserved, non-negativity is maintained, and the computed solutions reproduce the known analytic representations of the probability density functions associated with the L\'evy walk process. The combined theoretical and numerical framework provides a reliable tool for studying anomalous transport governed by fractional dynamics.
\end{abstract}



\section{Introduction}

Nonlocal operators arise in a wide range of contexts, from stochastic processes to partial differential equations, in both theoretical investigations and applied modeling. Among these, fractional-order derivatives, defined in several distinct ways, are perhaps the most prominent examples. They extended the notion of the classical integer-order derivative in multiple directions, motivated both by growing experimental evidence of their effectiveness in describing complex phenomena and by intrinsic mathematical interest. Naturally, such generalizations must be formulated carefully so as to remain well-defined and to preserve a sense of naturalness consistent with their classical counterparts.

The \emph{fractional material derivative} was introduced in \cite{sokolov2003towards} as the deterministic operator capable of describing the evolution of the probability density function of the  L\'evy walks (LW), a class of stochastic processes that capture anomalous transport. More precisely, the probability denisty function of the scaling limits of L\'evy walks is a mild solution of a partial differential equation involving a fractional material derivative (see section 2 in the article of Płociniczak and Teuerle \cite{Plociniczak2024} for more details on the link between L\'evy walks and fractional material derivative). Generally, for anomalous diffusion, the mean squared displacement (MSD) of a particle starting in origin deviates from the classical linear growth law, that is, $\langle x^{2}(t) \rangle \sim t$, and instead scales as
\begin{equation}
\langle x^{2}(t) \rangle \sim t^{\alpha}, \qquad \alpha \neq 1,
\end{equation}
see \cite{MetzlerKlafter}. When $0<\alpha<1$ the process is termed \emph{subdiffusion}, while the regime $1<\alpha<2$ corresponds to \emph{superdiffusion}. L\'evy walks provide a canonical model for the latter case, and the fractional material derivative offers a natural deterministic description of their macroscopic dynamics. LW's are stochastic processes in which a particle performs successive flights with lengths drawn from a heavy-tailed probability distribution, while each flight is traversed at a finite velocity. This coupling of step lengths and travel times distinguishes L\'evy walks from L\'evy flights and leads to superdiffusive spreading observed in many natural and physical systems. Anomalous dynamics arise naturally in the description of diffusion in heterogeneous or disordered environments and have been found in many fields such as porous media \cite{plociniczak2014approximation, plociniczak2015analytical, plociniczak2019derivation, pachepsky2000simulating, El20}, neurological tissues \cite{magin2010fractional}, polymers \cite{muller2011nonlinear}, single particle tracking in biophysics \cite{tabei2013intracellular, Sun17, wong2004anomalous}, plasma physics \cite{Del05}, astrophysics \cite{lawrence1993anomalous}, chemotaxis \cite{langlands2010fractional} and financial mathematics \cite{jacquier2020anomalous}. On the other hand, {superdiffusion} refers to processes in which particles spread faster than in classical diffusion, often due to long-range correlations or heavy-tailed jump distributions. Superdiffusive dynamics are frequently modeled by L\'evy flights and L\'evy walks, and they appear in diverse applications such as turbulent transport in fluids and plasmas \cite{shlesinger1993strange, zaburdaev2015levy}, animal foraging patterns \cite{viswanathan1999optimizing}, human travel and mobility \cite{brockmann2006scaling, gonzalez2008understanding}, ecology and biology \cite{humphries2014optimal, reynolds2009levy, alves2016transient}, light transport in random media \cite{mercadier2009levy, barthelemy2008levy}, and even in financial time series exhibiting heavy-tailed fluctuations \cite{mantegna1995scaling}.

Nonlocal differential equations are also interesting mathematical objects to study and investigate. They provide a deeper structure and higher difficulty of the corresponding equations. We do not want to give a complete overlook of the literature of fractional differential equations, but rather mention several important papers both in theory and numerical methods. The first comprehensive study of the solvability of the linear anomalous diffusion equation was carried out in \cite{Sak11} where the authors showed that a typical solution has a discontinuous time derivative at its origin. In subsequent years, the study of similar time-nonlocal PDEs evolved in various directions. For example, nonlinear diffusion equations were investigated, decay properties were studied, and different types of solutions were introduced \cite{allen2016parabolic, wittbold2021bounded, vergara2015optimal, akagi2019fractional, dipierro2019decay,plociniczak2018existence}. Along with an interesting mathematical structure that requires a careful setting in appropriate functional spaces, the numerical study of nonlocal equations also possesses many interesting hurdles. The main difficulties arise from the nonlocality of the differential operator, the lack of sufficient smoothness of the solution, and the more stringent stability conditions. However, a great deal of progress has been made in the last two decades. The interested reader is invited to consult the following paper (and references therein) for finite element approaches \cite{Jin19a, Mus18, plociniczak2022error,plociniczak2024fully, plociniczak2023linear}, spectral scheme \cite{Lin07}, and fast convolution quadrature methods \cite{lopez2025convolution, cuesta2006convolution}. In this paper, we construct a finite-volume scheme for our main equation since this approach is much more appropriate when one wants to preserve the mass of the solution (that is, the probability). Related results can be found in \cite{karaa2017finite, liu2014new} (see also a new interesting account on the L\'evy-Fokker-Planck equation with fractional Laplacian \cite{bailo2025finite}). 

Let $L_{loc}^1(\Omega)$ be the space of locally, that is, in compact sets $K\subseteq \Omega$, integrable functions on $\Omega$. Define \emph{fractional integral} of the order $\alpha>0$ of $y\in L_{loc}^1\left(\mathbb{R}_+\right)$ given by
\begin{equation}
I^{\alpha}y\left(t\right) = \frac{1}{\Gamma\left(\alpha\right)}\int_0^t\left(t-s\right)^{\alpha-1}y\left(s\right)ds.
\end{equation}
The main nonlocal PDE of our paper is composed as a convex combination of the fractional material derivative equated to a source term
\begin{equation}\label{combinationeq}
\begin{cases}
	p\left(\dfrac{\partial}{\partial t} - \dfrac{\partial}{\partial x}\right)^{\alpha} u\left(x, t\right) + \left(1-p\right)\left(\dfrac{\partial}{\partial t} + \dfrac{\partial}{\partial x}\right)^{\alpha} u\left(x, t\right)  = f\left(x,t\right), \vspace{2pt}\\ 
	\lim_{t \to 0^{+}}I_t^{1-\alpha}u\left(x, t\right) = g\left(x\right),
\end{cases}
\end{equation}
where $f\left(\cdot, t\right) \in L_{loc}^1\left(\mathbb{R}\right)$ for all $t\in\mathbb{R}_+$ and $f\left(x,\cdot\right) \in L_{loc}^1\left(\mathbb{R}_+\right)$ for all $x\in\mathbb{R}$, $g\in L^{1}_{loc}\left(\mathbb{R}\right)$ and {fractional material derivative} is defined by the Fourier-Laplace transform multiplier
\begin{equation}
\mathcal{F}\mathcal{L}\left\{\left(\dfrac{\partial}{\partial t} \pm \dfrac{\partial}{\partial x}\right)^{\alpha} u(x,t) \right\}
(\xi,s):=(s\mp i \xi)^\alpha \mathcal{F}\mathcal{L}\left\{u(x,t)\right\}(\xi,s),
\label{def:fracMatDer}
\end{equation}
provided that the above transforms exist. Equation \eqref{combinationeq} has been seen in \cite{sokolov2003towards} to arise as a general deterministic description of a one-dimensional L\'evy walk. An interesting fact proved in \cite{Plociniczak2024} that the fractional material derivative possesses a pointwise representation state in terms of the fractional \emph{Riemann-Liouville derivatives} defined by
\begin{equation}
D^{\alpha}y\left(t\right) = \dfrac{d^n}{dt^n}I^{n-\alpha}y\left(t\right), \quad n=\left[\alpha\right]+1,
\end{equation}
where $\left[\alpha\right]$ is the integer part of $\alpha$. In particular, for the physically interesting case with $0<\alpha<1$ we have
\begin{equation}
D^{\alpha}y\left(t\right) = \frac{1}{\Gamma\left(1-\alpha\right)}\dfrac{d}{dt}\int_0^t\left(t-s\right)^{-\alpha}y\left(s\right)ds.
\end{equation}
The result is as follows.
\begin{theorem}\label{pointwiserep}\cite{Plociniczak2024}
Let $u\left(\cdot, t\right) \in L_{loc}^1\left(\mathbb{R}\right)$ for all $t\in\mathbb{R}_+$ and $u\left(x,\cdot\right) \in L_{loc}^1\left(\mathbb{R}_+\right)$ for all $x\in\mathbb{R}$. Then, for $0<\alpha<1$ we have
\begin{equation}\label{pointwiserepeq}
	\left(\dfrac{\partial}{\partial t} \pm \dfrac{\partial}{\partial x}\right)^{\alpha} u\left(x, t\right) = \frac{1}{\Gamma\left(1-\alpha\right)}\left(\dfrac{\partial}{\partial t} \pm \dfrac{\partial}{\partial x}\right) \int_0^t \left(t-s\right)^{-\alpha} u\left(x \mp \left(t-s\right), s\right)ds.
\end{equation}
\end{theorem}
Not only does this form let us define the fractional material derivative under much weaker regularity assumptions, it also helps to devise efficient finite-volume schemes and even solve the $p=0,1$ equations explicitly (see \cite{Plociniczak2024}). In the following, we will use it to prove several other properties and construct a numerical method for \eqref{combinationeq}.

The main results of our paper are the following:
\begin{itemize}
    \item Existence and uniqueness of the mild solution of \eqref{combinationeq} for weak regularity of the data.
    \item Necessary and sufficient condition for the solution to be a probability density function.
    \item Convergent and conservative finite-volume numerical scheme that can be applied to \eqref{combinationeq}. 
\end{itemize}
To our knowledge, all of the above results are the first in the literature. They complete the theory of the one-dimensional deterministic description of the scaling limits of asymmetric L\'evy walks ($p\neq 0.5$).

Finally, we introduce the function spaces that we will work on below. Since the fractional material derivative is normally defined via the Fourier-Laplace multiplier as in \eqref{def:fracMatDer}, it is necessary to define the spaces in which it is well-defined. The space in which the Fourier transform naturally lives is the \emph{Schwartz space} of functions that decay faster than algebraically, more specifically, we define
\begin{equation}
\mathcal{S}\left(\mathbb{R}\right) = \left\{y\in C^{\infty}\left(\mathbb{R}\right): \text{for all } n, m \in \mathbb{N} \, \sup_{x \in \mathbb{R}} \left|x^n\dfrac{d^m y}{dx^n}\right|<\infty \right\}.
\end{equation}
On the other hand, the Laplace transform is well-defined on the space of functions of \emph{exponential type}, that is
\begin{equation}
X\left(\mathbb{R}_+\right) = \left\{y:\mathbb{R}_+ \mapsto \mathbb{R}: \text{ there exist constants } M, \, T_0 \in \mathbb{R}_+, \, a\in \mathbb{R} \text{ such that } \left|y\left(t\right)\right|\leq M e^{at} \text{ for } t\geq t_0\right\}.
\end{equation}
The Lebesgue space of functions integrable with power $1\leq p \leq \infty$ is denoted by $L^p(\Omega)$ with $\Omega \subseteq \mathbb{R}^d$ with $d\geq 1$. The space of continuous functions with the continuous $k$-th derivative is $C^{(k)}(\Omega)$ and, similarly, the space of functions that disappear at infinity is denoted by $C_0(\mathbb{R})$.

In the next section, we prove the existence and uniqueness of its solutions. Section 3 contains the construction of a conservative numerical scheme to solve our main problem. In Section 4, we present several numerical simulations that verify our theoretical findings.

\section{Existence and uniqueness}
For completeness of the discussion concerning the fractional material derivative, we prove some of its basic properties that, to our knowledge, have not been reported in the literature. First, we formulate a simple remark about scaling of fractional material derivatives, which can be helpful in future work about self-similar solutions of equations involving fractional material derivatives. This clearly shows that the fractional material derivative is indeed a differential operator of order $\alpha>0$.
\begin{proposition}
Define $u_{\lambda}\left(x, t\right) := u\left(\lambda x, \lambda t\right)$, $\lambda>0$. Then,
\begin{equation}
	\left(\left(\dfrac{\partial}{\partial t} \pm \dfrac{\partial}{\partial x}\right)^{\alpha} u_{\lambda}\right)\left(x, t\right) = \lambda^{\alpha}\left(\left(\dfrac{\partial}{\partial t} \pm \dfrac{\partial}{\partial x}\right)^{\alpha} u\right)\left(\lambda x, \lambda t\right).
\end{equation}
\end{proposition}
\begin{proof}
Using pointwise representation \eqref{pointwiserepeq}, from direct calculations we have 
\begin{equation}
\begin{split}
	\left(\left(\dfrac{\partial}{\partial t} \pm \dfrac{\partial}{\partial x}\right)^{\alpha} u_{\lambda}\right)\left(x, t\right) &= \frac{1}{\Gamma\left(1-\alpha\right)}\left(\dfrac{\partial}{\partial t} \pm \dfrac{\partial}{\partial x}\right) \int_0^t \left(t-s\right)^{-\alpha} u_{\lambda}\left(x \mp \left(t-s\right), s\right)ds \\ 
	&= \frac{1}{\Gamma\left(1-\alpha\right)}\left(\dfrac{\partial}{\partial t} \pm \dfrac{\partial}{\partial x}\right) \int_0^t \left(t-s\right)^{-\alpha} u\left(\lambda \left(x \mp \left(t-s\right)\right), \lambda s\right)ds.
\end{split}
\end{equation}
Now, by a change of the variable $v = \lambda s$, we have
\begin{equation}
\begin{split}
	\left(\left(\dfrac{\partial}{\partial t} \pm \dfrac{\partial}{\partial x}\right)^{\alpha} u_{\lambda}\right)\left(x, t\right) &=\frac{1}{\Gamma\left(1-\alpha\right)}\left(\dfrac{\partial}{\partial t} \pm \dfrac{\partial}{\partial x}\right) \int_0^{\lambda t} \left(t-\frac{v}{\lambda}\right)^{-\alpha} u\left(\lambda \left(x \mp t\right)\pm v, v\right)\dfrac{dv}{\lambda} \\ 
	&= \frac{\lambda^{\alpha}}{\Gamma\left(1-\alpha\right)}\left(\dfrac{\partial}{\partial \lambda t} \pm \dfrac{\partial}{\partial \lambda x}\right) \int_0^{\lambda t} \left(\lambda t-v\right)^{-\alpha} u\left(\lambda \left(x \mp t\right)\pm v, v\right)dv \\ &= \lambda^{\alpha}\left(\left(\dfrac{\partial}{\partial t} \pm \dfrac{\partial}{\partial x}\right)^{\alpha} u\right)\left(\lambda x, \lambda t\right),
\end{split}
\end{equation}
which completes the proof. 
\end{proof}
Note how the pointwise representation helped to deliver the above result in a straightforward way. Next, we establish simple connection between fractional material derivative with $+$ and $-$.
\begin{proposition}\label{minusremark}
Let $u\left(\cdot, t\right) \in L_{loc}^1\left(\mathbb{R}\right)$ for all $t\in\mathbb{R}_+$ and $u\left(x,\cdot\right) \in L_{loc}^1\left(\mathbb{R}_+\right)$ for all $x\in\mathbb{R}$  and $y=-x$. Then
\begin{equation}
	\left(\dfrac{\partial}{\partial t} \pm \dfrac{\partial}{\partial x}\right)^{\alpha} u\left(x, t\right) = \left(\dfrac{\partial}{\partial t} \mp \dfrac{\partial}{\partial y}\right)^{\alpha} u\left(-y, t\right).
\end{equation}
\end{proposition}
\begin{proof}
Using pointwise representation \eqref{pointwiserepeq}, from direct calculations we have 
\begin{equation}
	\begin{aligned}
		\left(\dfrac{\partial}{\partial t} \mp \dfrac{\partial}{\partial y}\right)^{\alpha} u\left(-y, t\right) = \frac{1}{\Gamma\left(1-\alpha\right)}\left(\dfrac{\partial}{\partial t} \mp \dfrac{\partial}{\partial y}\right) \int_0^t \left(t-s\right)^{-\alpha} u\left(-y \mp \left(t-s\right), s\right)ds = \\ = \frac{1}{\Gamma\left(1-\alpha\right)}\left(\dfrac{\partial}{\partial t} \pm \dfrac{\partial}{\partial x}\right) \int_0^t \left(t-s\right)^{-\alpha} u\left(x \mp \left(t-s\right), s\right)ds = \left(\dfrac{\partial}{\partial t} \pm \dfrac{\partial}{\partial x}\right)^{\alpha} u\left(x, t\right),
	\end{aligned}
\end{equation}
and the proof is complete. 
\end{proof}

After the above simple statements, we now proceed to the main analytical result of this paper. In the next theorem, we prove the existence and uniqueness of solutions to \eqref{combinationeq}.
\begin{theorem}\label{exuniq}
Let $f\left(\cdot, t\right) \in C_0\left(\mathbb{R}\right)$ for all $t\in\mathbb{R}_+$ and $f\left(x,\cdot\right) \in X\left(\mathbb{R}_+\right)$ for all $x\in\mathbb{R}$,  $g\in C_0\left(\mathbb{R}\right)$. Then there exists a unique mild solution to the problem \eqref{combinationeq}. Morevover it is give by the Duhamel formula
\begin{equation}\label{duhamel}
	u\left(x,t\right) = G_{t}(x)* g(x) + \int_0^t (G_{t-\tau}* f)\left(x,\tau\right)d\tau, 
\end{equation}
where $G_t\left(x\right) = \mathcal{F}^{-1}\left\{\mathcal{L}^{-1}\left\{\frac{1}{p\left(s- i\xi\right)^{\alpha} + \left(1-p\right)\left(s+ i\xi\right)^{\alpha}}\right\}\left(t\right)\right\}\left(x\right)$.
\begin{proof}
We split the proof into two parts. First, we assume that $f\left(\cdot, t\right) \in \mathcal{S}\left(\mathbb{R}\right)$ for all $t\in\mathbb{R}_+$ and $f\left(x,\cdot\right) \in X\left(\mathbb{R}_+\right)$ for all $x\in\mathbb{R}$ , $g\in \mathcal{S}\left(\mathbb{R}\right)$. Let us denote by $\hat{u}\left(\xi,t\right),\, \hat{f}\left(\xi,t\right),\, \hat{g}\left(\xi\right)$ Fourier transforms of $u, \, f, \, g$ in space. By $U\left(\xi, s\right), \, F\left(\xi, s\right)$ we denote the Fourier-Laplace transforms of $u, \, f$. It is known that the Laplace transform of the Riemann-Liouville derivative is given by
\begin{equation}
	\mathcal{L}\left\{D^{\alpha}f\left(t\right)\right\}\left(s\right) = s^{\alpha}\mathcal{L}\left\{f\left(t\right)\right\}\left(s\right) - \left[I_t^{1-\alpha}f\left(t\right)\right]_{t=0}.
\end{equation}
Using this fact, we calculate the Fourier-Laplace transform of the problem \eqref{combinationeq}
\begin{equation}
	p\left[\left(s- i\xi\right)^{\alpha}U\left(\xi, s\right) - \hat{g}\left(\xi\right)\right] + \left(1-p\right)\left[\left(s+ i\xi\right)^{\alpha}U\left(\xi, s\right) - \hat{g}\left(\xi\right)\right] = F\left(\xi,s\right),
\end{equation}
and after rearranging, we obtain
\begin{equation}\label{flteq}
	U\left(\xi, s\right) = \frac{F\left(\xi, s\right) + \hat{g}\left(\xi\right)}{p\left(s- i\xi\right)^{\alpha} + \left(1-p\right)\left(s+ i\xi\right)^{\alpha}}.
\end{equation}
Now, we will prove that we can invert the Laplace transform. We will use the Bromwich integral
\begin{equation}
	\mathcal{L}^{-1}\left\{U\left(\xi, s\right)\right\}\left(t\right) = \frac{1}{2\pi i}\int_{\gamma-i\infty}^{\gamma+i\infty}e^{st}U\left(\xi, s\right)ds,
\end{equation}
to show that the transforms can be inverted, yielding a unique (due to linearity) mild solution of our problem. This integral is well-defined when :
\begin{itemize}
    \item the Laplace transform is an analytic function when for $\Re\left(s\right)>\gamma$, for some $\gamma \in \mathbb{R}$,
    \item decays sufficiently fast as $\left|s\right|\to\infty$ in this region.
\end{itemize}
Since $f$ is exponential in time, its Laplace transform is analytic for $\Re\left(s\right)>a$ for some $a\in\mathbb{R}$. This property carries over to the Fourier transform; thus, $F\left(\xi, s\right)$ is $s$-analytic for $\Re\left(s\right)>a$. Moreover, $\hat{g}\left(\xi\right)$ is constant with respect to $s$, thus the numerator of \eqref{flteq} is analytic. The denominator is a combination of functions of the form $\left(s\pm i\xi\right)^{\alpha}$. The complex power function $z^{\alpha}$ is analytic everywhere except the non-positive real axis. Let $s=\sigma+i\omega$. For any $s$ in the right half-plane $\Re\left(s\right)=\sigma>0$, the real parts of $s\pm i\xi$ are equal to $\sigma$, which is positive. Thus, both components of the denominator are analytic for $\Re\left(s\right)>0$. 

Furthermore, the denominator does not have zeros for this region. To show this, let us assume that there exists a zero $s=s_0$ of the denominator. Then,
\begin{equation}
	\left(\frac{s_0-i\xi}{s_0+i\xi}\right)^\alpha = -\frac{1-p}{p}. 
\end{equation}
The right-hand side of this equality is a negative real number, so its argument is
\begin{equation}
	\arg\left(-\frac{1-p}{p}\right) = \left(2n+1\right)\pi, \quad n\in\mathbb{Z}.
\end{equation}
Let us denote $w = (s_0-i\xi)/(s_0+i\xi)$. Thus, taking the argument, we have
\begin{equation}
	\arg\left(w^{\alpha}\right) = \alpha\arg\left(w\right).
\end{equation}
Hence,
\begin{equation}
	\arg\left(w\right) = \frac{\left(2n+1\right)\pi}{\alpha}.
\end{equation}
Now, let us calculate the argument of $w$ directly. We have $s_0 = \sigma + i\omega$, $\sigma>0$ and
\begin{equation}
\arg\left(w\right)=\arg\left(\frac{s_0-i\xi}{s_0+i\xi}\right) = \arg\left(\sigma + i\left(\omega-\xi\right)\right) -  \arg\left(\sigma + i\left(\omega+\xi\right)\right) = \arctan\left(\frac{\omega - \xi}{\sigma}\right) - \arctan\left(\frac{\omega + \xi}{\sigma}\right).
\end{equation} Thus, 
\begin{equation}
	-\pi<\arg\left(w\right)<\pi,
\end{equation}
becasue the range of $\arctan$ is between $-\pi$ and $\pi$. Combining it all together, we have
\begin{equation}
	-\alpha<2n+1<\alpha.
\end{equation}
There is no integer $n$, which can satisfy this relation, thus the denominator is always nonzero. Hence, we conclude that $U\left(s, \xi\right)$, as the ratio of analytic functions with a zero-free denominator in the region $\Re\left(s\right)>\max\left\{0, a\right\}$, is analytic.

The above allows us to place the Bromwich contour $\gamma$ to the right of the line $\max\left\{0, a\right\}$. Next, we show that $U\left(\xi, s\right)$ decays sufficiently fast in order for the integral to be convergent. Again, let $s = \gamma+i\omega$. Then, we have
\begin{equation}
	p\left(s- i\xi\right)^{\alpha} + \left(1-p\right)\left(s+ i\xi\right)^{\alpha} = \left(i\omega\right)^{\alpha}\left[p\left(1+\frac{\gamma-i\xi}{i\omega}\right)^{\alpha} + \left(1-p\right)\left(1+\frac{\gamma+i\xi}{i\omega}\right)^{\alpha}\right] = \left(i\omega\right)^{\alpha}B\left(\xi,\omega\right).
\end{equation}
Taking the absolute value, we have
\begin{equation}
	\left|\left(i\omega\right)^{\alpha}B\left(\xi,\omega\right)\right|=\left|\omega\right|^{\alpha}\left|B\left(\xi,\omega\right)\right|.
\end{equation}
From the Taylor expansion for the function $\left(1+z\right)^{\alpha}$, we have
\begin{equation}\label{sigmadecay}
	\left|\omega\right|^{\alpha}\left|B\left(\xi,\omega\right)\right| = \left|\omega\right|^{\alpha}\left|1 + \mathcal{O}\left(\frac{1}{\omega}\right)\right|.
\end{equation}
Thus, the denominator is $\mathcal{O}\left(\left|s\right|\right)^{\alpha}$. For the numerator, we have
\begin{equation}
	\left|F\left(\xi, s\right) + \hat{g}\left(\xi\right)\right|\leq \left|F\left(\xi, s\right)\right| + \left|\hat{g}\left(\xi\right)\right|.
\end{equation}
For $\left|F\left(\xi, s\right)\right|$, we have
\begin{equation}
	\left|F\left(\xi, s\right)\right| = \left|\int_0^{\infty}e^{-st}\hat{f}\left(\xi,t\right)dt\right|\leq\int_0^{\infty}\left|e^{-st}\right|\left|\hat{f}\left(\xi,t\right)\right|dt.
\end{equation}
Since $s = \gamma+i\omega$, we have $\left|e^{-st}\right| = e^{-\gamma t}$. Since $f\left(x,\cdot\right)\in X\left(\mathbb{R}_+\right)$, we have
\begin{equation}
	\left|F\left(\xi, s\right)\right| \leq \int_0^{\infty}e^{-\gamma t}Me^{at} = \frac{M}{\gamma-a}.
\end{equation}
Therefore, $\left|F\left(\xi, s\right)\right|$ is bounded by a constant independent of $s$. Combining it all together, we have
\begin{equation}
	\left|U\left(\xi, s\right)\right| = \mathcal{O}\left(\left|s\right|^{-\alpha}\right) \quad \text{as} \quad \Re{s} \rightarrow \infty,
\end{equation}
which is the decay we need for the Bromwich inversion to be justified. Since $\hat{u}\left(\cdot, t\right)\in \mathcal{S}\left(\mathbb{R}\right)$, we can also invert the Fourier transform, thus there exists a unique mild solution to equation \eqref{combinationeq}.

Now, let us choose two sequences $g_n\in\mathcal{S}\left(\mathbb{R}\right)$ with $g_n\to g$ uniformly and $f_n\left(\cdot, t\right) \in C_0\left(\mathbb{R}\right)$ for all $t\in\mathbb{R}_+$ and $f_n\left(x,\cdot\right) \in X\left(\mathbb{R}_+\right)$ for all $x\in\mathbb{R}$ with $f_n\to f$ uniformly in $x$ and in the norm of space $X$ in $t$. Then for each $\left(f_n, g_n\right)$ there exists a unique solution, $u_n$, such that $u_n\left(\cdot, t\right) \in C_0\left(\mathbb{R}\right)$ for all $t\in\mathbb{R}_+$ and $u_n\left(x,\cdot\right) \in X\left(\mathbb{R}_+\right)$ for all $x\in\mathbb{R}$. $u_n$ can be represented as
\begin{equation}
	u_n\left(x,t\right) = \int_{\mathbb{R}}G_t\left(x-y\right)g_n\left(y\right)dy + \int_0^t (G_{t-\tau}* f_n)\left(x,\tau\right)d\tau, 
\end{equation}
where 
\begin{equation}\label{flg}
	G_t\left(x\right) = \mathcal{F}^{-1}\left\{\mathcal{L}^{-1}\left\{\frac{1}{\zeta\left(\xi, s\right)}\right\}\left(t\right)\right\}\left(x\right),
\end{equation}
where $\zeta\left(\xi, s\right) = p\left(s- i\xi\right)^{\alpha} + \left(1-p\right)\left(s+ i\xi\right)^{\alpha}$. Let us denote
\begin{equation}
	P\left(\xi, t\right)= \mathcal{L}^{-1}\left\{\frac{1}{\zeta\left(\xi, s\right)}\right\}\left(t\right) = \frac{1}{2\pi i }\int_{\gamma-i\infty}^{\gamma+i\infty}e^{st}\frac{1}{\zeta\left(\xi, s\right)}ds.
\end{equation}
Now, we change the variables $s = \lambda/t$, $\gamma = \gamma'/t$ to obtain
\begin{equation}\label{laplaceaux}
	P(\xi,t) = \frac{1}{2\pi i} \int_{\gamma' - i\infty}^{\gamma' + i\infty} e^{(\lambda/t)t} \frac{1}{\zeta(\xi, \lambda/t)} \frac{d\lambda}{t}.
\end{equation}
We also note that
\begin{equation}
	\zeta\left(\xi, \lambda/t\right) = p\left(\frac{\lambda}{t} - i\xi\right)^\alpha + \left(1-p\right)\left(\frac{\lambda}{t} + i\xi\right)^\alpha = \left(\frac{1}{t}\right)^\alpha \left[ p\left(\lambda - i\xi t\right)^\alpha + \left(1-p\right)\left(\lambda + i\xi t\right)^\alpha \right] = t^{-\alpha} \zeta\left(\xi t, \lambda\right),
\end{equation}
which together with \eqref{laplaceaux} yields
\begin{equation}
	P(\xi,t) = t^{\alpha-1}\left(\frac{1}{2\pi i}\int_{\gamma' - i\infty}^{\gamma' + i\infty} \frac{e^{\lambda}}{\zeta(\xi t, \lambda)}d\lambda\right)=t^{\alpha-1}{\mathcal{L}^{-1}\left\{\frac{1}{\zeta\left(\xi t, \lambda\right)}\right\}\left(1\right)}.
\end{equation}
By substituting $P$ into \eqref{flg}, we obtain
\begin{equation}
	G_t\left(x\right) = \mathcal{F}^{-1}\left\{P(\xi,t)\right\} = \frac{1}{2\pi}\int_{\mathbb{R}}e^{i\xi x}t^{\alpha-1}{\mathcal{L}^{-1}\left\{\frac{1}{\zeta\left(\xi t, \lambda\right)}\right\}\left(1\right)}d\xi.
\end{equation}
Now, we change the variables $y=\xi t$. We have,
\begin{equation}
\begin{split}
    G_t(x) &= t^{\alpha-1} \frac{1}{2\pi} \int_{-\infty}^{\infty} e^{i(y/t)x} \left( \mathcal{L}^{-1}\left\{\frac{1}{\zeta(y, \lambda)}\right\}(1) \right) \frac{dy}{t} = t^{\alpha-2} \left( \frac{1}{2\pi} \int_{-\infty}^{\infty} e^{iy(x/t)} \left( \mathcal{L}^{-1}\left\{\frac{1}{\zeta(y, \lambda)}\right\}(1) \right) dy \right) \\ 
    &= t^{\alpha-2}\mathcal{F}^{-1}\left\{\mathcal{L}^{-1}\left\{\frac{1}{\zeta\left(y, \lambda\right)}\right\}\left(1\right)\right\}\left(
    \frac{x}{t}\right) = t^{\alpha-2}G_1\left(\frac{x}{t}\right).
\end{split}
\end{equation}
Next, we estimate the supremum norm of $u_n$. Using the integral representation of the solution, we have
\begin{equation}
	\left\|u_n\left(\cdot, t\right)\right\|_{\infty} \leq\left\|\int_{\mathbb{R}}G_t\left(x-y\right)g_n\left(y\right)dy \right\|_{\infty} + \left\|\int_0^tG_{t-\tau}\left(\cdot\right)*f_n\left(\cdot,\tau\right)d\tau\right\|_{\infty}
\end{equation}
By Young's convolution inequality $\left\|f*g\right\|_{L^r}\leq\left\|f\right\|_{L^p}\left\|g\right\|_{L^q}$ for $1/p + 1/q = 1 + 1/r$, $1\leq p,q,r\leq\infty$ (we can apply it, since the $C_0 \subset L^{\infty}$) and the scaling property, we have
\begin{equation}
	\left\|\int_{\mathbb{R}}G_t\left(x-y\right)g_n\left(y\right)dy \right\|_{\infty} \leq \left\|G_t\right\|_{L^1}\left\|g_n\right\|_{\infty} = t^{\alpha-1} \|G_1\|_{L^1}\left\|g_n\right\|_{\infty}.
\end{equation}
Similarly, for the second integral
\begin{equation}
\begin{split}
    &\left\|\int_0^tG_{t-\tau}\left(\cdot\right)*f_n\left(\cdot,\tau\right)d\tau\right\|_{\infty}\leq\int_0^t\left\|G_{t-\tau}\left(\cdot\right)*f_n\left(\cdot,\tau\right)\right\|_{\infty} d\tau \\ 
    &\leq\int_0^t\left\|G_{t-\tau}\left(\cdot\right)\right\|_{L^1}\left\|f_n\left(\cdot,\tau\right)\right\|_{\infty} d\tau = \|G_1\|_{L^1}\int_0^t\left(t-\tau\right)^{\alpha-1}\left\|f_n\left(\cdot,\tau\right)\right\|_{\infty} d\tau
\end{split}
\end{equation}
We have to show that $\|G_1\|_{L^1}<\infty$. To this end, write 
\begin{equation}
	P\left(\xi\right) = \mathcal{L}_{\lambda}^{-1}\left\{\frac{1}{\zeta\left(\xi, \lambda\right)}\right\}\left(1\right) = \frac{1}{2\pi i}\int_{\gamma-i\infty}^{\gamma+i\infty}\frac{e^{\lambda}}{\zeta\left(\xi, \lambda\right)}d\lambda.
\end{equation} 
Notice that $P\left(\xi\right)$ is smooth since $\zeta\left(\xi, \lambda\right)$ is differentiable and positive on the integration path $\gamma$. 

Now, we need to show that we can switch derivative and integral. We can do this using the Leibniz rule.
Let $\lambda = \gamma + i\eta$. Then, we have
\begin{equation}
	\left| \dfrac{\partial^n}{\partial \xi^n}\left(\frac{1}{\zeta}\right) \right| = \mathcal{O}\left(\left|\eta\right|^{-\left(\alpha+n\right)}\right)
\end{equation}
To show this, we proceed by induction. For $n=0$, it follows from \eqref{sigmadecay}. Now, let us assume that the induction hypothesis is true for all $j=0, 1, \ldots, n-1$. Then, for $n$, we have
\begin{equation}\label{derivbehav}
	\left|\dfrac{\partial^n}{\partial \xi^n}\left(\frac{1}{\zeta}\right)\right| = \left|-\frac{1}{\zeta}\right| \cdot \left| \sum_{j=1}^{n} \binom{n}{j} \dfrac{\partial^j \zeta}{\partial \xi^j} d\frac{\partial^{n-j}}{\partial \xi^{n-j}}\left(\frac{1}{\zeta}\right) \right| = \mathcal{O}\left(\left|\eta\right|^{-\alpha}\right)\mathcal{O}\left(\left|\eta\right|^{-n}\right)=\mathcal{O}\left(\left|\eta\right|^{-\left(\alpha+n\right)}\right).
\end{equation}
Thus, the induction hypothesis is true. 

Hence, we can differentiate $P\left(\xi\right)$ under the integral sign and it is a smooth function. Now, let us rewrite $P\left(\xi\right)$ in the following way
\begin{equation}
	P\left(\xi\right) = P\left(\xi\right)\phi\left(\xi\right) + P\left(\xi\right)\left(1-\phi\left(\xi\right)\right),
\end{equation}
where $\phi\left(\xi\right)$ is the cut-off function, that is, a smooth function equal to $1$ for $\left|\xi\right|\leq1$ and equal to $0$ for $\left|\xi\right|\geq2$. Let $P_{low}\left(\xi\right) = P\left(\xi\right)\phi\left(\xi\right )$, $P_{high}\left(\xi\right) = P\left(\xi\right)\left(1-\phi\left(\xi\right)\right)$,
\begin{equation}
	G_{low}\left(x\right) = \frac{1}{2\pi}\int_{\mathbb{R}}e^{i\xi x}P_{low}\left(\xi\right)d\xi,
\end{equation}
and
\begin{equation}
	G_{high}\left(x\right) = \frac{1}{2\pi}\int_{\mathbb{R}}e^{i\xi x}P_{high}\left(\xi\right)d\xi.
\end{equation}
Using integration by parts twice, we have
\begin{equation}
	G_{low}\left(x\right) = \frac{1}{2\pi x^2}\int_{\mathbb{R}}e^{i\xi x}\left(\dfrac{d^2}{d\xi^2}P_{low}\left(\xi\right)\right)d\xi.
\end{equation}
Taking the absolute value on both sides, we obtain
\begin{equation}
	\left|G_{low}\left(x\right)\right|\leq\frac{1}{2\pi x^2}\int_{\mathbb{R}}\left|\dfrac{d^2}{d\xi^2}P_{low}\left(\xi\right)\right|d\xi.
\end{equation}
Since $P_{low}$ is smooth on a compact interval and all its
derivatives are bounded and the integral on the right-hand side is a finite constant $C_1$. Thus, we have
\begin{equation}
	\left|G_{low}\left(x\right)\right| \leq \frac{C_1}{x^2}.
\end{equation}
Thus, $\left\|G_{low}\right\|_{L_1}<\infty$. For $G_{high}$, again integrating by parts twice, we have
\begin{equation}
	G_{high}\left(x\right) = \frac{1}{2\pi x^2}\int_{\mathbb{R}}e^{i\xi x}\left(\dfrac{d^2}{d\xi^2}P_{high}\left(\xi\right)\right)d\xi.
\end{equation}
Taking the absolute value on both sides, we obtain
\begin{equation}
	\left|G_{high}\left(x\right)\right|\leq\frac{1}{2\pi x^2}\int_{\mathbb{R}}\left|\dfrac{d^2}{d\xi^2}P_{high}\left(\xi\right)\right|d\xi.
\end{equation}
From \eqref{derivbehav} we know that $\left|\dfrac{d^2}{d\xi^2}P_{high}\left(\xi\right)\right|$ behaves like $\left|\xi\right|^{-\left(\alpha+2\right)}$ at infinity, thus the right-hand side of the above inequality converges, and we have
\begin{equation}
	\left|G_{high}\left(x\right)\right| \leq \frac{C_2}{x^2}.
\end{equation}
Finally, since $G_1\left(x\right) = G_{low}\left(x\right) + G_{high}\left(x\right)$, we have
\begin{equation}
	\left\|G_1\right\|_{L^1}\leq \left\|G_{low}\right\|_{L_1} + \left\|G_{high}\right\|_{L_1} < \infty.
\end{equation}
Let us denote $C = \left\|G_1\right\|_{L^1}$. Then, we have
\begin{equation}
	\left\|u_n\left(\cdot, t\right)\right\|_{\infty}\leq Ct^{\alpha-1} \left\|g_n\right\|_{\infty} 
	+ C\int_0^t\left(t-\tau\right)^{\alpha-1}\left\|f_n\left(\cdot,\tau\right)\right\|_{\infty} d\tau.
\end{equation}
Having this estimate, we can prove that $u_n\left(\cdot, t\right)$ is the Cauchy sequence. That is, we have
\begin{equation}
	\left\|u_n\left(\cdot, t\right) - u_m\left(\cdot, t\right)\right\|_{\infty}\leq Ct^{\alpha-1} \left\|g_n-g_m\right\|_{\infty} 
	+ C\int_0^t\left(t-\tau\right)^{\alpha-1}\left\|f_n\left(\cdot,\tau\right)-f_m\left(\cdot,\tau\right)\right\|_{\infty} d\tau.
\end{equation}
Since $g_n$ and $f_n$ are Cauchy sequences, the right-hand side of the above inequality converges to $0$. Thus, left-hand side also converges to zero and $u_n$ is a Cauchy sequence. Since $\mathcal{S}$ is dense in $C_0$, $u_n\to u$ is uniformly in space. Hence, there exists a unique solution $u$, such that $u\left(\cdot, t\right) \in C_0\left(\mathbb{R}\right)$ for all $t\in\mathbb{R}_+$ and $u\left(x,\cdot\right) \in X\left(\mathbb{R}_+\right)$ for all $x\in\mathbb{R}$. 
\end{proof}    
\end{theorem}
Before, we state next theorem, we will prove important fact about kernel $G_t$.
\begin{proposition}
    Let $G_t\left(x\right) = \mathcal{F}^{-1}\left\{\mathcal{L}^{-1}\left\{\frac{1}{p\left(s- i\xi\right)^{\alpha} + \left(1-p\right)\left(s+ i\xi\right)^{\alpha}}\right\}\left(t\right)\right\}\left(x\right)$. Then,
    \begin{equation}
        \int_{-\infty}^{\infty}G_t\left(x\right)dx=\frac{t^{\alpha-1}}{\Gamma\left(\alpha\right)}.
    \end{equation}
    \begin{proof}
        By the definition of $G_t$, we have
    \begin{equation}
        G_t\left(x\right) = \frac{1}{2\pi}\int_{-\infty}^{\infty}e^{i\xi x}\mathcal{L}^{-1}\left\{\frac{1}{p\left(s- i\xi\right)^{\alpha} + \left(1-p\right)\left(s+ i\xi\right)^{\alpha}}\right\}\left(t\right)d\xi = \frac{1}{2\pi}\int_{-\infty}^{\infty}e^{i\xi x}P\left(\xi, t\right)d\xi.
    \end{equation}
    Integrating over space, we obtain
\begin{equation}
    \begin{aligned}
        \int_{-\infty}^{\infty}G_t\left(x\right)dx = \frac{1}{2\pi}\int_{-\infty}^{\infty}\int_{-\infty}^{\infty}e^{i\xi x}P\left(\xi, t\right)d\xi dx = \int_{-\infty}^{\infty}P\left(\xi, t\right)\left[\frac{1}{2\pi}\int_{-\infty}^{\infty}e^{i\xi x}dx\right] d\xi = \\ \int_{-\infty}^{\infty}P\left(\xi, t\right)\delta\left(\xi\right) d\xi = P\left(0, t\right).
    \end{aligned}
    \end{equation}
    Thus
    \begin{equation}
    \begin{aligned}
        \int_{-\infty}^{\infty}G_t\left(x\right)dx = P\left(0, t\right) = \mathcal{L}^{-1}\left\{\frac{1}{p\left(s- i0\right)^{\alpha} + \left(1-p\right)\left(s+ i0\right)^{\alpha}}\right\}\left(t\right) = \mathcal{L}^{-1}\left\{\frac{1}{s^{\alpha}}\right\}\left(t\right) = \frac{t^{\alpha-1}}{\Gamma\left(\alpha\right)}.
    \end{aligned}
    \end{equation}
    \end{proof}
\end{proposition}
Since we are interested in connections between \eqref{combinationeq} and the corresponding stochastic process in the next proposition, we find the necessary and sufficient condition on the source $f$ for the solution to be the probability density function for all times.
\begin{proposition}\label{probconsprop}
Let $u\left(x, t\right)$ be the solution to \eqref{combinationeq}, vanishing for $x\rightarrow\pm \infty$. Moreover, assume that $f\left(\cdot, t\right) \in L^1\left(\mathbb{R}\right)$ for all $t\in\mathbb{R}_+$ and $f\left(x,\cdot\right) \in L_{loc}^1\left(\mathbb{R}_+\right)$ for all $x\in\mathbb{R}$. Then, solution $u$ to the problem \eqref{combinationeq} is a probability density function, i.e.
 \begin{equation}\label{densitycond}
     \int_{-\infty}^{\infty}u(x,t)dx=1.
 \end{equation}
 and
 \begin{equation}
     \lim_{t\to0^+}u(x,t)=\delta(x)
 \end{equation}
 if only if
\begin{equation}\label{pdfcond}
	\int_{-\infty}^{\infty}f\left(x, t\right)dx = \frac{t^{-\alpha}}{\Gamma\left(1-\alpha\right)}.
\end{equation}
Moreover, the solution is given by
\begin{equation}\label{duhamel2}
    u\left(x,t\right) = \int_0^t (G_{t-\tau}* f)\left(x,\tau\right)d\tau.
\end{equation}
\begin{proof}
Using the fact that $L^1$ space is dense in Schwartz space, following the same reasoning as in proof of Theorem \ref{exuniq}, we obtain that $u$ is the solution to \eqref{combinationeq}. Let us assume that condition \eqref{densitycond} holds. Integrating the equation Duhamel formula \eqref{duhamel2} over $\mathbb{R}$ with respect to $x$ we have
\begin{equation}\label{eqn:MainEqIntegrated}
\begin{split}
	1 =  \int_{-\infty}^{\infty}\int_0^t \int_{-\infty}^{\infty} G_{t-\tau}\left(x-y\right)f\left(y,\tau\right)dyd\tau dx
\end{split}
\end{equation}
Using the fact, that $\int_{-\infty}^{\infty}G_t\left(x\right)dx = \frac{t^{\alpha-1}}{\Gamma\left(\alpha\right)}$, we have
\begin{equation}\label{eqn:MainEqIntegrated}
\begin{split}
	1 = \frac{1}{\Gamma\left(\alpha\right)}\int_0^t \int_{-\infty}^{\infty} \left(t-\tau\right)^{\alpha-1}f\left(y,\tau\right)dyd\tau = I^{1-\alpha}_t\int_{-\infty}^{\infty}f\left(x,t\right)dx.
\end{split}
\end{equation}
Now, assume that condition \eqref{pdfcond} holds. One more time integrating the Duhamel formula in $\mathbb{R}$ with respect to $x$, we have
\begin{equation}\label{eqn:MainEqIntegrated}
\begin{split}
	\int_{-\infty}^{\infty}u\left(x,t\right)dx =  \int_{-\infty}^{\infty}\int_0^t \int_{-\infty}^{\infty} G_{t-\tau}\left(x-y\right)f\left(y,\tau\right)dyd\tau dx.
\end{split}
\end{equation}
Again, using the fact that $\int_{-\infty}^{\infty}G_t\left(x\right)dx = \frac{t^{\alpha-1}}{\Gamma\left(\alpha\right)}$, we have
\begin{equation}
\begin{split}
	\int_{-\infty}^{\infty}u\left(x,t\right)dx = \frac{1}{\Gamma\left(1-\alpha\right)\Gamma\left(\alpha\right)}\int_0^t \left(t-\tau\right)^{\alpha-1}\tau^{-\alpha}d\tau = 1.
\end{split}
\end{equation}
We are left with showing the convergence to the Dirac delta. To this end, take any take any $\psi\in C_c^{\infty}\left(\mathbb{R}\right)$, multiply it by \eqref{duhamel2}, and integrate to arrive at
       \begin{equation}\label{weakformulation2}
           \int_{-\infty}^{\infty}u\left(x,t\right)\psi\left(x\right)dx = \int_{-\infty}^{\infty}\psi\left(x\right)\left(\int_0^t (G_{t-\tau}* f)\left(x,\tau\right)d\tau\right)dx.
       \end{equation}
       Furthermore, by adding and subtracting $\psi\left(0\right)$ using Fubini's theorem, the fact, that $\int_{-\infty}^{\infty}G_t\left(x\right)dx = \frac{t^{\alpha-1}}{\Gamma\left(\alpha\right)}$, using \eqref{pdfcond} and utilizing definition of beta fucntion, we obtain
       \begin{equation}
           \begin{aligned}
               \int_{-\infty}^{\infty}\psi\left(x\right)\left(\int_0^t (G_{t-\tau}* f)\left(x,\tau\right)d\tau\right)dx = \int_0^t\int_{-\infty}^{\infty}f(y,\tau)(G_{t-\tau}* (\psi-\psi(0)))\left(y,\tau\right)dyd\tau + \\ \psi(0)\int_0^t\int_{-\infty}^{\infty}f(y,\tau)\left(\int_{-\infty}^{\infty}G_{t-\tau}(x-y)dx\right)dyd\tau = \psi(0)+\int_0^t\int_{-\infty}^{\infty}f(y,\tau)(G_{t-\tau}* (\psi-\psi(0)))\left(y,\tau\right)dyd\tau
           \end{aligned}
       \end{equation}
       Hence, we are left in showing that the double integral on the right-hand side vanishes as $t\to0^+$.  To see this, change the time variable $\tau=ts$, $y=tv$, $z=t(1-s)w$ and use scaling of kernel $G_t(x) = t^{\alpha-2}G_1(x/t)$ to have
       \begin{equation}
           \begin{aligned}
               \left|\int_0^t\int_{-\infty}^{\infty}\int_{-\infty}^{\infty}f(y,\tau)G_{t-\tau}(z)(\psi(y-z)-\psi(0))dzdyd\tau\right|\\ \leq t^{\alpha+1}\int_0^1(1-s)^{\alpha-1}\int_{-\infty}^{\infty}\int_{-\infty}^{\infty}f(tv, ts)G_1(w)|\psi(t(v-(1-s)w))-\psi(0)|dwdvds\\ \leq \sup_{\xi\in \mathbb{R}}\left|\psi\left(t\xi\right)-\psi\left(0\right)\right|t^{\alpha+1}\int_0^1(1-s)^{\alpha-1}\int_{-\infty}^{\infty}\int_{-\infty}^{\infty}f(tv, ts)G_1(w)dwdvds
           \end{aligned}
       \end{equation}
       Now, using fact, thats space integral of kernel $G_t$ is equal to $t^{\alpha-1}/\Gamma(\alpha)$ and since condition \eqref{pdfcond} is satisfied we obtain
       \begin{equation}
           \begin{aligned}
               \left|\int_0^t\int_{-\infty}^{\infty}\int_{-\infty}^{\infty}f(y,\tau)G_{t-\tau}(z)(\psi(y-z)-\psi(0))dzdyd\tau\right|\\ \leq \sup_{\xi\in \mathbb{R}}\left|\psi\left(t\xi\right)-\psi\left(0\right)\right|\frac{t^{\alpha+1}}{\Gamma\left(\alpha\right)}\int_0^1\left(1-s\right)^{\alpha-1}\left(\int_{-\infty}^{\infty}f\left(tv,ts\right)dv\right)ds \\ = t\sup_{\xi\in \mathbb{R}}\left|\psi\left(t\xi\right)-\psi\left(0\right)\right|\frac{1}{\Gamma\left(\alpha\right)\Gamma\left(1-\alpha\right)}\int_0^1\left(1-s\right)^{\alpha-1}s^{-\alpha}dz = t\sup_{\xi \in \mathbb{R}}\left|\psi\left(t\xi\right)-\psi\left(0\right)\right|
           \end{aligned}
       \end{equation}
       By uniform convergence of $\psi$ the right-hand side of the above goes to $0$ as $t\to0^+$.
\end{proof}
\end{proposition}

It is interesting that there exist several types of L\'evy walks that yield solutions $u=u(x,t)$ in an exact, analytical form. These explicit formulas were derived in \cite{Magdziarz2020} using probabilistic methods.  

\begin{example}\label{wfex} For the \emph{wait first} L\'evy walk, we have \cite{MagdziarzTeuerle2015}
\begin{equation}
\begin{cases}
	p\left(\dfrac{\partial}{\partial t} - \dfrac{\partial}{\partial x}\right)^{\alpha} u\left(x, t\right) + \left(1-p\right)\left(\dfrac{\partial}{\partial t} + \dfrac{\partial}{\partial x}\right)^{\alpha} u\left(x, t\right)  = \dfrac{t^{-\alpha}}{\Gamma\left(1-\alpha\right)}\delta\left(x\right) \vspace{2pt}\\ 
	\lim_{t \to 0^{+}}u\left(x, t\right) = \delta\left(x\right).
\end{cases}
\end{equation}
The solution has the explicit self-similar form $u\left(x,t\right) = \frac{1}{t}\phi\left(\frac{x}{t}\right)$, where
\begin{equation}
\phi\left(y\right) = 
\begin{cases}
	\dfrac{p\sin \left(\alpha \pi\right)}{\pi}\dfrac{\left(1-y\right)^{\alpha}y^{-1+\alpha}}{p^2\left(1-y\right)^{2\alpha} + \left(1-p\right)^2\left(1+y\right)^{2\alpha}+2p\left(1-p\right)\left(1-y^2\right)^{\alpha}\cos\left(\alpha \pi\right)}, & 0<y<1, \vspace{2pt}\\ 
	\dfrac{\left(1-p\right)\sin \left(\alpha \pi\right)}{\pi}\dfrac{\left(1+y\right)^{\alpha}\left(-y\right)^{-1+\alpha}}{\left(1-p\right)^2\left(1-y\right)^{2\alpha} + p^2\left(1+y\right)^{2\alpha}+2p\left(1-p\right)\left(1-y^2\right)^{\alpha}\cos\left(\alpha \pi\right)}, & -1<y<0, \vspace{2pt}\\
	0, \text{otherwise}.
\end{cases}
\end{equation}
\end{example}

\begin{example} In the case of the \emph{jump first} L\'evy walk \cite{Teuerle2012}, the main equation has the form \cite{MagdziarzTeuerle2015}
\begin{equation}
\begin{split}
\begin{cases}
	p\left(\dfrac{\partial}{\partial t} - \dfrac{\partial}{\partial x}\right)^{\alpha} u\left(x, t\right) + \left(1-p\right)&\left(\dfrac{\partial}{\partial t} + \dfrac{\partial}{\partial x}\right)^{\alpha} u\left(x, t\right) \\ &=\dfrac{\alpha}{\Gamma\left(1-\alpha\right)} \displaystyle{\int_t^{\infty} \left(p\delta\left(x+u\right) + \left(1-p\right)\delta\left(x-u\right)\right) u^{-\alpha-1}du}\\ 
	\lim_{t \to 0^{+}}u\left(x, t\right) = \delta\left(x\right).
\end{cases}
\end{split}
\end{equation}
with the solution $u\left(x,t\right) = \frac{1}{t}\phi\left(\frac{x}{t}\right)$, where
\begin{equation}
\phi\left(y\right) = 
\begin{cases}
	\dfrac{p\sin \left(\alpha \pi\right)}{\pi y}\dfrac{1}{p\left(y-1\right)^{\alpha} + \left(1-p\right)\left(y+1\right)^{\alpha}}, & y > 1, \vspace{2pt}\\ 
	\dfrac{\left(1-p\right)\sin \left(\alpha \pi\right)}{-\pi y}\dfrac{1}{\left(1-p\right)\left(-y-1\right)^{\alpha} + p\left(-y+1\right)^{\alpha}}, & y < -1, \vspace{2pt}\\
	\dfrac{p\left(1-p\right)\sin \left(\alpha \pi\right)}{\pi\left|y\right|}\dfrac{\left|\left(1+y\right)^{\alpha} - \left(1-y\right)^{\alpha}\right|}{p^2\left(1-y\right)^{2\alpha} + \left(1-p\right)^2\left(1+y\right)^{2\alpha}+2p\left(1-p\right)\left(1-y^2\right)^{\alpha}\cos\left(\alpha \pi\right)}, & \text{ otherwise}.
\end{cases}
\end{equation}
	
\end{example}
\begin{example} Now let us consider standard Levy walk. Then, the governing equation is of the form\cite{MagdziarzEtAl2015LevyWalks}
\begin{equation}
\begin{cases}
	p\left(\dfrac{\partial}{\partial t} - \dfrac{\partial}{\partial x}\right)^{\alpha} u\left(x, t\right) + \left(1-p\right)\left(\dfrac{\partial}{\partial t} + \dfrac{\partial}{\partial x}\right)^{\alpha} u\left(x, t\right)  = \dfrac{t^{-\alpha}}{\Gamma\left(1-\alpha\right)}\left[p\delta\left(x+t\right) + \left(1-p\right)\delta\left(x-t\right)\right] \vspace{2pt}\\ 
	\lim_{t \to 0^{+}}u\left(x, t\right) = \delta\left(x\right).
\end{cases}
\end{equation}
and the solution has the explicit form $u\left(x,t\right) = \frac{1}{t}\phi\left(\frac{x}{t}\right)$, where
\begin{equation}
	\phi\left(y\right) = 
	\frac{p\left(1-p\right)\sin \left(\alpha \pi\right)}{\pi }\frac{\left(1-y\right)^{\alpha-1}\left(1+y\right)^{\alpha} +\left(1+y\right)^{\alpha-1}\left(1-y\right)^{\alpha}}{p^2\left(1-y\right)^{2\alpha} + \left(1-p\right)^2\left(1+y\right)^{2\alpha}+2p\left(1-p\right)\left(1-y^2\right)^{\alpha}\cos\left(\alpha \pi\right)}H\left(t-\left|x\right|\right),
\end{equation}
with $H=H(t)$ being the Heaviside step function. 
\end{example}

\section{Numerical method}
In this section, we devise a numerical method to solve \eqref{combinationeq}. We will use the method obtained in \cite{Plociniczak2024} for the equation with a fractional material derivative. For completeness, we recall the derivation of this scheme. For the time- and space-mesh step $h>0$ (since time and space are treated on equal footing, this is a reasonable assumption), we define
\begin{equation}
	x_i = ih, \quad t_i = nh, \quad i\in \mathbb{Z}, \quad n \in \mathbb{N}
\end{equation}
and we partition the real line into finite cells (volumes)
\begin{equation}
	\mathbb{R} = \bigcup_{i=-\infty}^{\infty} \left(x_{i-\frac{1}{2}}, x_{i + \frac{1}{2}}\right], \quad x_{i\pm \frac{1}{2}} = \left(i\pm \frac{1}{2}\right)h.
\end{equation}
Now, we integrate Equation \eqref{pointwiserepeq} into an arbitrary cell. We have
\begin{equation}
\begin{split}
    \frac{1}{\Gamma\left(1-\alpha\right)}&\dfrac{d}{d t}\int_0^t \left(t-s\right)^{-\alpha} \left(\int_{x_{i-\frac{1}{2}} \mp\left(t-s\right)}^{x_{i+\frac{1}{2}} \mp \left(t-s\right)}u\left(y, s\right)dy\right)ds  \\ 
    &- \frac{1}{\Gamma\left(1-\alpha\right)}\int_0^t \left(t-s\right)^{-\alpha} \left(u\left(x_{i + \frac{1}{2}} \mp \left(t-s\right), s\right) - u\left(x_{i - \frac{1}{2}} \mp \left(t-s\right), s\right)\right)ds \\  &= \int_{x_{i - \frac{1}{2}}}^{x_{i + \frac{1}{2}}}f\left(x, t\right)dx
\end{split}
\end{equation}
We divide the equation by $h$ and define the cell-average
\begin{equation}
	\bar{u}\left(x_i, t\right) = \frac{1}{h}\int_{x_i - \frac{1}{2}}^{x_i + \frac{1}{2}} u\left(x, t\right)dx.
\end{equation}
Then we have
\begin{equation}\label{discretizedconslaw}
\begin{split}
    \frac{1}{\Gamma\left(1-\alpha\right)}&\dfrac{d}{d t}\int_0^t \left(t-s\right)^{-\alpha}\bar{u}\left(x_i \mp \left(t-s\right), s\right)ds \\ 
    &-  \frac{1}{\Gamma\left(1-\alpha\right)}\int_0^t \left(t-s\right)^{-\alpha} \frac
    {u\left(x_{i + \frac{1}{2}} \mp \left(t-s\right), s\right) - u\left(x_{i - \frac{1}{2}} \mp \left(t-s\right), s\right)}{h}ds  =  \bar{f}\left(x_i, t\right).
\end{split}
\end{equation}
Now, we discretize the temporal variable using the L1 scheme, that is, we approximate the unknown function by a linear interpolation in time, and integrate (see \cite{plociniczak2023linear} for a complete account). Thus, 
\begin{equation}
\begin{split}
    \frac{1}{\Gamma\left(1-\alpha\right)}\int_0^{t_n}& \left(t_n-s\right)^{-\alpha} \bar{u}\left(x_i \mp\left(t_n-s\right), s\right)ds = \frac{1}{\Gamma\left(1-\alpha\right)}\sum_{j=0}^{n-1}\int_{t_j}^{t_{j+1}} \left(t_n-s\right)^{-\alpha} \bar{u}\left(x_i \mp\left(t_n-s\right), s\right)ds \\ 
    &= \frac{h^{1-\alpha}}{\Gamma\left(2-\alpha\right)}\sum_{j=0}^{n-1}b_{n-j}u^j_{i\mp\left(n-j\right)} + r_{in}\left(h\right),
\end{split}
\end{equation}
where $r_{in}\left(h\right)$ is a remainder, $u_i^k := \bar{u}\left(x_i, t_k\right)$ and 
\begin{equation}\label{bdef}
	b_k = k^{1-\alpha} - \left(k-1\right)^{1-\alpha}.
\end{equation}
Therefore, 
\begin{equation}
\begin{split}
	\frac{1}{\Gamma\left(1-\alpha\right)}\dfrac{d}{dt}&\int_0^{t_n} \left(t_n-s\right)^{-\alpha} \bar{u}\left(x_i \mp\left(t_n-s\right), s\right)ds \\
    &= \frac{h^{-\alpha}}{\Gamma\left(2-\alpha\right)}\left(u_{i \mp1}^n+\sum_{j=0}^{n-1}b_{n-j+1}u^j_{i\mp\left(n-j+1\right)}-b_{n-j}u^j_{i\mp\left(n-j\right)}\right) + R_{in}^1\left(h\right),
\end{split}
\end{equation}
where $R_{im}^1\left(h\right)$ is a remainder and for sufficiently smooth functions, say twice differentiable this approximation is of order $2-\alpha$, that is remainder satisfies \cite{LiZeng2015NumericalMethodsFractionalCalculus}
\begin{equation}
    \left|R_{im}^1\right|\leq Ch^{2-\alpha}.
\end{equation}
Now, we can proceed to discretization of the spatial variable. We assume that the pointwise value of $u$ can be approximated by its cell average, that is,
\begin{equation}
	\bar{u}\left(x_i, t\right) \approx \begin{cases}
		u\left(x_{i + \frac{1}{2}}, t\right), \quad \text{in the "+" case,}\vspace{2pt}\\
		u\left(x_{i - \frac{1}{2}}, t\right), \quad \text{in the "-" case,}
	\end{cases}
\end{equation}
which is accurate with $O\left(h\right)$ as $h\to 0^+$. Thus, the discretization of the spatial part is
\begin{equation}
	\begin{aligned}
		\frac{h^{-1}}{\Gamma\left(1-\alpha\right)}\int_0^t \left(t-s\right)^{-\alpha} \left(u\left(x_{i + \frac{1}{2}} \mp \left(t-s\right), s\right) - u\left(x_{i - \frac{1}{2}} \mp \left(t-s\right), s\right)\right)ds = \\ = \frac{h^{-\alpha}}{\Gamma\left(2-\alpha\right)}\left(\sum_{j=0}^{n-1}b_{n-j}\left(u^j_{i\mp\left(n-j\right)}-u^j_{i\mp\left(n-j+1\right)}\right)\right) + R_{in}^2\left(h\right)
	\end{aligned}
\end{equation}
Truncating the remainders and combining all the above formulas together, we obtain the discretization form of \eqref{discretizedconslaw}
\begin{equation}\label{eqn:MatDerDisc}
	\delta^{\alpha}_{\pm}u_i^n = \frac{h^{-\alpha}}{\Gamma\left(2-\alpha\right)}\left[u_{i\mp1}^n - \sum_{j=0}^{n-1}\left(b_{n-j} - b_{n-j+1}\right)u_{i\mp\left(n-j+1\right)}^j\right].
\end{equation}
Suppose we would like to find the solution of
\begin{equation}\label{simplecombinationeq}
\begin{split}
\begin{cases}
    p\left(\dfrac{\partial}{\partial t} - \dfrac{\partial}{\partial x}\right)^{\alpha} u\left(x, t\right) + \left(1-p\right)\left(\dfrac{\partial}{\partial t} + \dfrac{\partial}{\partial x}\right)^{\alpha} u\left(x, t\right)  = f\left(x, t\right), \quad x \in \mathbb{R}, \, 0\leq t\leq T \\ 
    u\left(x, 0\right) = 0, \quad x\in \mathbb{R}.
\end{cases}
\end{split}
\end{equation}
The discretized version above has the form
\begin{equation}\label{combinationscheme}
\begin{split}
\begin{cases}
    u_i^n = \displaystyle{\sum_{j=0}^{n-1}\left(b_{n-j} - b_{n-j+1}\right)\left(pu_{i+\left(n-j\right)}^j + \left(1-p\right)u_{i-\left(n-j\right)}^j\right) +  h^{\alpha}\Gamma\left(2-\alpha\right)\left(pf_{i-1}^n + \left(1-p\right)f_{i+  1}^n\right)},\\
    u_i^0 = 0.
\end{cases}
\end{split}
\end{equation}
Let us show some fundamental properties of the scheme \eqref{combinationscheme}. We start by showing that it is stable. Introduce the discrete $p$-th norm
\begin{equation}\label{discretenorm}
	\left\|u^n\right\|_{p,h} = \left(h\sum_{i \in \mathbb{Z}} \left|u_i^n\right|^p\right)^{\frac{1}{p}}, \quad 1\leq p < \infty. 
\end{equation}
The below result shows that the second norm of the solution is controlled by the norm of the data. 
\begin{proposition}[Stability]\label{prop:Stability}
Let $u_m^n$ be the solution of \eqref{combinationscheme} for $f\left(\cdot, t\right)\in L^2\left(\mathbb{R}\right)$ for each $t \in \left[0, T\right]$, $f\left(x, \cdot\right)\in L^{\infty}\left(0, T\right)$ for each $x \in \mathbb{R}$ and $h\leq\left(1-\alpha\right)^{1/2\alpha}$. Then, we have
\begin{equation}
\max_n\left\|u^n\right\|_{2,h} \leq T^\alpha \Gamma(1-\alpha)\max_n\left\|f^n\right\|_{2,h}.
\end{equation}
\begin{proof}
We start by multiplying \eqref{combinationscheme} by $hu_i^n$ and summing over $i$ to obtain
\begin{equation}\label{eqn:StabilityEst}
\begin{split}
    \left\|u^n\right\|_{2,h}^2 &= h\sum_{i \in \mathbb{Z}}\left[\sum_{j=0}^{n-1}\left(b_{n-j} - b_{n-j+1}\right)\left(pu_{i+\left(n-j\right)}^j u_i^n + \left(1-p\right)u_{i-\left(n-j\right)}^ju_i^n\right) \right. \\
    &\left.+ h^{\alpha}\Gamma\left(2-\alpha\right)\left(pf^n_{i-1}u_i^n + \left(1-p\right)f^n_{i+1}u_i^n \right)\right].
\end{split}
\end{equation}
Now we estimate the first and second term of the outer sum. For the former, using the Cauchy inequality ($ab \leq \left(a^2+b^2\right)/2$), we have
\begin{equation}
    \begin{aligned}
        h\sum_{i \in \mathbb{Z}}\left[\sum_{j=0}^{n-1}\left(b_{n-j} - b_{n-j+1}\right)\left(pu_{i+\left(n-j\right)}^j u_i^n + \left(1-p\right)u_{i-\left(n-j\right)}^ju_i^n\right)\right] \\ \leq\frac{h}{2}\sum_{i \in \mathbb{Z}}\left[\sum_{j=0}^{n-1}\left(b_{n-j} - b_{n-j+1}\right)\left(p\left(\left(u_{i+\left(n-j\right)}^j\right)^2 +  \left(u_i^n\right)^2\right) + \left(1-p\right)\left(\left(u_{i-\left(n-j\right)}^j\right)^2 + \left(u_i^n\right)^2\right)\right)\right].
    \end{aligned}
\end{equation}
Next, we change the order of summation and split the sum over $j$ into three different ones
\begin{equation}
    \begin{aligned}
        h\sum_{i \in \mathbb{Z}}\left[\sum_{j=0}^{n-1}\left(b_{n-j} - b_{n-j+1}\right)\left(pu_{i+\left(n-j\right)}^j u_i^n + \left(1-p\right)u_{i-\left(n-j\right)}^ju_i^n\right)\right]  \\ \leq \frac{1}{2}\sum_{j=0}^{n-1}\left(b_{n-j} - b_{n-j+1}\right)h\sum_{i \in \mathbb{Z}}\left(u_i^n\right)^2 + \frac{p}{2}\sum_{j=0}^{n-1}\left(b_{n-j} - b_{n-j+1}\right)h\sum_{i \in \mathbb{Z}}\left(u_{i+\left(n-j\right)}^j\right)^2 \\ + \frac{1-p}{2}\sum_{j=0}^{n-1}\left(b_{n-j} - b_{n-j+1}\right)h\sum_{i \in \mathbb{Z}}\left(u_{i-\left(n-j\right)}^j\right)^2.
    \end{aligned}
\end{equation}
Finally, using the definition of the norm \eqref{discretenorm}, the fact that $b_{n-j} - b_{n-j+1}\geq0$, and the series $\sum_{j=0}^{n-1}\left(b_{n-j} - b_{n-j+1}\right)$ is telescoping, we obtain
\begin{equation}
    \begin{aligned}
        h\sum_{i \in \mathbb{Z}}\left[\sum_{j=0}^{n-1}\left(b_{n-j} - b_{n-j+1}\right)\left(pu_{i+\left(n-j\right)}^j u_i^n + \left(1-p\right)u_{i-\left(n-j\right)}^ju_i^n\right)\right] \\ \leq \frac{1}{2}\left\|u^n\right\|_{2,h}^2\sum_{j=0}^{n-1}\left(b_{n-j} - b_{n-j+1}\right) + \frac{p}{2}\sum_{j=0}^{n-1}\left(b_{n-j} - b_{n-j+1}\right)\left\|u^j\right\|_{2,h}^2  \\ + \frac{1-p}{2}\sum_{j=0}^{n-1}\left(b_{n-j} - b_{n-j+1}\right)\left\|u^j\right\|_{2,h}^2  = \frac{1}{2}\left(1 - b_{n+1}\right)\left\|u^n\right\|_{2,h}^2 + \frac{1}{2}\sum_{j=0}^{n-1}\left(b_{n-j} - b_{n-j+1}\right)\left\|u^j\right\|_{2,h}^2.
    \end{aligned}
\end{equation}
For the second term in \eqref{eqn:StabilityEst} using Cauchy-Schwarz inequality, we obtain
\begin{equation}
\begin{split}
    h\sum_{i \in \mathbb{Z}}  &h^{\alpha}\Gamma\left(2-\alpha\right)\left(pf^n_{i-1}u_i^n + \left(1-p\right)f^n_{i+1}u_i^n \right) \leq h\sum_{i \in \mathbb{Z}}  h^{\alpha}\Gamma\left(2-\alpha\right)\left(p\left|f^n_{i-1}\right|\left|u_i^n\right| + \left(1-p\right)\left|f^n_{i+1}\right|\left|u_i^n\right| \right) \\ 
    &\leq h^\alpha \Gamma\left(2-\alpha\right)\left(p\sqrt{h\sum_{i \in \mathbb{Z}}\left|f^n_{i-1}\right|^2} \sqrt{h\sum_{i \in \mathbb{Z}}\left|u_i^n\right|^2} + \left(1-p\right)\sqrt{h\sum_{i \in \mathbb{Z}}\left|f^n_{i+1}\right|^2}\sqrt{h\sum_{i \in \mathbb{Z}}\left|u_i^n\right|^2} \right)  \\
    &= h^{\alpha}\Gamma\left(2-\alpha\right)\left\|f^n\right\|_{2,h}\left\|u^n\right\|_{2,h}.
\end{split}
\end{equation}
Now, combining the above inequalities, we have
\begin{equation}\label{firstineqstab}
    \left\|u^n\right\|_{2,h}^2 \leq \frac{1}{2}\left(1 - b_{n+1}\right)\left\|u^n\right\|_{2,h}^2 + \frac{1}{2}\sum_{j=0}^{n-1}\left(b_{n-j} - b_{n-j+1}\right)\left\|u^j\right\|_{2,h}^2 + h^{\alpha}\Gamma\left(2-\alpha\right)\left\|f^n\right\|_{2,h}\left\|u^n\right\|_{2,h},
\end{equation}
Now, using the $\varepsilon-$Cauchy inequality ($ab \leq \left(a^2/\varepsilon+\varepsilon b^2\right)/2$) to the last term on the right-hand side with $a=\left\|f^n\right\|_{2,h}$, $b=\left\|u^n\right\|_{2,h}$, and $\varepsilon = b_{n+1}h^{-\alpha}/\Gamma\left(2-\alpha\right)$ we have 
\begin{equation}
    \left\|u^n\right\|_{2,h}^2 \leq \sum_{j=0}^{n-1}\left(b_{n-j} - b_{n-j+1}\right)\left\|u^j\right\|_{2,h}^2 + \frac{\left[h^{\alpha}\Gamma\left(2-\alpha\right)\right]^2}{b_{n+1}}\left\|f^n\right\|_{2,h}^2.
\end{equation}
And after taking the maximum with respect to $n$, using the telescoping property of the series, and observing that due to \eqref{bdef} the coefficient $b_{n+1}$ is decreasing, we can obtain
\begin{equation}
    \max_j\left\|u^j\right\|_{2,h}^2 \leq \left(b_1 - b_{n+1}\right)\max_j\left\|u^j\right\|_{2,h}^2 + \frac{\left[h^{\alpha}\Gamma\left(2-\alpha\right)\right]^2}{b_{n+1}}\max_j\left\|f^j\right\|_{2,h}^2.
\end{equation}
Since, by \eqref{bdef} we have $b_1 = 1$, we can further simplify both sides to
\begin{equation}
    \max_j\left\|u^j\right\|_{2,h} \leq \frac{h^{\alpha}\Gamma\left(2-\alpha\right)}{b_{n+1}}\max_j\left\|f^j\right\|_{2,h}.
\end{equation}
But, by the mean-value theorem there exist such $n - 1 \leq \xi \leq n$, that
\begin{equation}
    \frac{h^\alpha \Gamma(2-\alpha)}{b_{n+1}} = \frac{h^\alpha \Gamma(2-\alpha)}{n^{1-\alpha}- (n-1)^{1-\alpha}} = \frac{h^{\alpha}\Gamma(2-\alpha)}{(1-\alpha) \xi^{-\alpha}} \leq (nh)^\alpha \Gamma(1-\alpha) \leq T^\alpha \Gamma(1-\alpha),
\end{equation}
so that the multiplicative constant is bounded. The proof is complete. 
\end{proof}
\end{proposition}

Next, we will prove the convergence of the scheme. 
\begin{theorem}[Convergence]\label{convthm}
Let $u\left(x, t\right)$ be the solution of \eqref{simplecombinationeq} and $u_i^n$ be the solution of the numerical scheme \eqref{combinationscheme}. Assume that $f\left(\cdot, t\right)\in L^2\left(\mathbb{R}\right)$, $f\left(x, \cdot\right)\in L^{\infty}\left(0, T\right)$ for each $x \in \mathbb{R}$ and that the truncation error satisfies
\begin{equation}\label{truncassumption}
    \left\| \delta_{\pm}^{\alpha}u^n - \left(\dfrac{\partial}{\partial t} \pm \dfrac{\partial}{\partial x}\right)^{\alpha} u\left(\cdot, t_n\right)\right\|_{2,h} = \left\|R^n\left(h\right)\right\|_{2,h} \leq Ch^r, \quad r>0.
\end{equation}
Then
\begin{equation}
    \left\|u\left(\cdot, t_n\right) - u^n\right\|_{2,h} \leq T^\alpha \Gamma(1-\alpha) C h^r,
\end{equation}
where constant $C$ depends on $\alpha, \, u$ and its derivatives, while order $r>0$ is related to the regularity of the solution.
\end{theorem}
\begin{proof}
We define the error $e_i^n = u\left(x_i, t_n\right) - u^n_i$ and apply the operator $p\left(\dfrac{\partial}{\partial t} - \dfrac{\partial}{\partial x}\right)^{\alpha}+ \left(1-p\right)\left(\dfrac{\partial}{\partial t} + \dfrac{\partial}{\partial x}\right)^{\alpha}$. Then,
\begin{equation}
\begin{split}
    p&\left(\dfrac{\partial}{\partial t} - \dfrac{\partial}{\partial x}\right)^{\alpha} e_i^n + \left(1-p\right)\left(\dfrac{\partial}{\partial t} + \dfrac{\partial}{\partial x}\right)^{\alpha} e_i^n \\
    &=  p\left(\dfrac{\partial}{\partial t} - \dfrac{\partial}{\partial x}\right)^{\alpha} \left(u\left(x_i, t_n\right) - u^n_i\right) + \left(1-p\right)\left(\dfrac{\partial}{\partial t} + \dfrac{\partial}{\partial x}\right)^{\alpha} \left(u\left(x_i, t_n\right) - u^n_i\right).
\end{split}
\end{equation}
Using \eqref{simplecombinationeq}, adding and subtracting $\delta^{\alpha}u_i^n = p\delta^{\alpha}_-u_i^n + q\delta^{\alpha}_+u_i^n$ from the above equation, we have
\begin{equation}
\begin{split}
    p\left(\dfrac{\partial}{\partial t} - \dfrac{\partial}{\partial x}\right)^{\alpha} e_i^n &+ \left(1-p\right)\left(\dfrac{\partial}{\partial t} + \dfrac{\partial}{\partial x}\right)^{\alpha} e_i^n \\ 
    &=f(x_i,t_n) - p\left(\dfrac{\partial}{\partial t} - \dfrac{\partial}{\partial x}\right)^{\alpha}u_i^n - \left(1-p\right)\left(\dfrac{\partial}{\partial t} + \dfrac{\partial}{\partial x}\right)^{\alpha} u_i^n + \delta^{\alpha}u_i^n- \delta^{\alpha}u_i^n 
\end{split}
\end{equation}
Now, combining the respective operators acting on $u^n_i$ and using the definition of discrete fractional material derivatives \eqref{eqn:MatDerDisc} and the scheme, we further obtain
\begin{equation}
\begin{split}
    p\left(\dfrac{\partial}{\partial t} - \dfrac{\partial}{\partial x}\right)^{\alpha} e_i^n &+ \left(1-p\right)\left(\dfrac{\partial}{\partial t} + \dfrac{\partial}{\partial x}\right)^{\alpha} e_i^n \\ 
    &=f\left(x_i, t_n\right) - p\left[\left(\dfrac{\partial}{\partial t} - \dfrac{\partial}{\partial x}\right)^{\alpha} - \delta_-^{\alpha}\right] u_i^n - \left(1-p\right)\left[\left(\dfrac{\partial}{\partial t} + \dfrac{\partial}{\partial x}\right)^{\alpha} - \delta_+^{\alpha}\right] u_i^n - f\left(x_i, t_n\right).
\end{split}
\end{equation}
Hence, we see that the error $e_i^n$ satisfies \eqref{simplecombinationeq} with the source equal to a convex combination of the truncation error of the discrete fractional material derivative. Therefore, from \eqref{truncassumption}, our assumption \eqref{truncassumption}, and Proposition \ref{prop:Stability}, we have
\begin{equation}
\begin{split}
    \left\|e_i^n\right\|_{2,h} &\leq T^\alpha \Gamma(1-\alpha) \left\|p\left[\delta_-^{\alpha}-\left(\dfrac{\partial}{\partial t} - \dfrac{\partial}{\partial x}\right)^{\alpha}\right] u^n + \left(1-p\right)\left[\delta_+^{\alpha}+\left(\dfrac{\partial}{\partial t} - \dfrac{\partial}{\partial x}\right)^{\alpha}\right]  u^n\right\|_{2,h}\\ 
    &\leq T^\alpha \Gamma(1-\alpha) h^r,
\end{split}
\end{equation}
what ends the proof.
\end{proof}

In the context of L\'evy walks, the solution of \eqref{combinationeq} is a probability \eqref{combinationeq}. It is natural to require the same property from the discrete approximation, at least to some error dependent on the grid spacing $h>0$. From Proposition \ref{probconsprop} we know that conservation of probability implies that the source possesses a certain singularity. Numerically speaking, it would not be optimal to construct our schemes based on high-order discretizations. We are thus free to modify our method as long as we stay within the same order of the truncation error. In the next result, we will show that if we compute the source at an advanced time $t = t_{n+1}$, we can easily show the discrete conservation of probability. Our scheme for \eqref{combinationeq} is thus the following
\begin{equation}\label{combinationscheme2}
	\begin{cases}
		u_i^n = \sum_{j=0}^{n-1}\left(b_{n-j} - b_{n-j+1}\right)\left(pu_{i+\left(n-j\right)}^j + \left(1-p\right)u_{i-\left(n-j\right)}^j\right) + h^{\alpha}\Gamma\left(2-\alpha\right)\left(pf_{i-1}^{n+1} + \left(1-p\right)f_{i+  1}^{n+1}\right), \\
		u_i^0 = \psi\left(x_i\right),
	\end{cases}
\end{equation}
where $\psi$ is the numerical initial condition. Below, we prove that it is conservative. 
\begin{theorem}
Let $f$ be a non-negative function satisfying \eqref{pdfcond} with
\begin{equation}\label{discefcond}
\left\|f^n\right\|_{1, h} = \frac{t_n^{-\alpha}}{\Gamma\left(1-\alpha\right)} + \rho\left(h\right),
\end{equation}
with the quadrature error $\rho\left(h\right)$ and $t_n = nh\leq T$. Then, if $\left\|u^0\right\|_{1,h}=1$ we have
\begin{equation}
\left\|u^n\right\|_{1, h} \leq 1 + T\Gamma\left(2-\alpha\right)h^{\alpha-1}\rho\left(h\right).
\end{equation}
\begin{proof}
Because $f\geq0$, from \eqref{combinationscheme2} we have $u_j^n\geq 0$. Since $b_{n-j}-b_{n-j+1}\geq0$, multiplying \eqref{combinationscheme2} by $h$ and summing over $i$, we obtain
\begin{equation}
\begin{split}
    \left\|u^n\right\|_{1, h} &= \sum_{j=0}^{n-1}\left(b_{n-j}-b_{n-j+1}\right)\left(p\left\|u^j\right\|_{1, h}+\left(1-p\right)\left\|u^j\right\|_{1, h}\right) + h^{\alpha}\Gamma\left(2-\alpha\right)\left(p\left\|f^{n+1}\right\|_{1, h}+\left(1-p\right)\left\|f^{n+1}\right\|_{1, h}\right) \\ 
    &= \sum_{j=0}^{n-1}\left(b_{n-j}-b_{n-j+1}\right)\left\|u^j\right\|_{1, h} + h^{\alpha}\Gamma\left(2-\alpha\right)\left\|f^{n+1}\right\|_{1, h}.
\end{split}
\end{equation}
Now, from \eqref{discefcond}, $\left\|u_0\right\|_{1,h}=1$ and fact, that $t_{n+1} = \left(n+1\right)h$, we have
\begin{equation}
    \left\|u^n\right\|_{1, h} = b_n - b_{n+1} +\sum_{j=1}^{n-1}\left(b_{n-j}-b_{n-j+1}\right)\left\|u^j\right\|_{1, h} + \frac{1-\alpha}{\left(n+1\right)^{\alpha}}+ h^{\alpha}\Gamma\left(2-\alpha\right)\rho\left(h\right).
\end{equation}
Now, we will prove by induction that
\begin{equation}\label{indas}
    \left\|u^n\right\|_{1, h} \leq 1 + nh^{\alpha}\Gamma\left(2-\alpha\right)\rho\left(h\right),
\end{equation}
because then,
\begin{equation}
    \left\|u^n\right\|_{1, h} \leq 1 + nhh^{\alpha-1}\Gamma\left(2-\alpha\right)\rho\left(h\right) \leq 1 + Th^{\alpha-1}\Gamma\left(2-\alpha\right)\rho\left(h\right),
\end{equation}
what is our claim. For $n=1$, from \eqref{bdef} we have
\begin{equation}
    \left\|u^1\right\|_{1, h} \leq 2 - 2^{1-\alpha}+ \frac{1-\alpha}{2^{1-\alpha}}+h^{\alpha}\Gamma\left(2-\alpha\right)\rho\left(h\right)
\end{equation}
By elementary calculus one can show that the function $\alpha \mapsto 2 - 2^{1-\alpha}+ (1-\alpha)(2^{\alpha-1}$ is positive and bounded from below by $2\left(\log 2 - e^{-1}\right)/\log 2$ and by $1$ from above. Hence
\begin{equation}
    \left\|u^1\right\|_{1, h} \leq 1 + h^{\alpha}\Gamma\left(2-\alpha\right)\rho\left(h\right).
\end{equation}
Now, assume that \eqref{indas} holds for all $k=1, 2, \ldots, n-1$. For $n$, we then have
\begin{equation}
    \left\|u^1\right\|_{1, h} \leq b_n - b_{n+1} +\left(1 + \left(n-1\right)h^{\alpha}\Gamma\left(2-\alpha\right)\rho\left(h\right)\right)\sum_{j=0}^{n-1}\left(b_{n-j}-b_{n-j+1}\right)+ \frac{1-\alpha}{\left(n+1\right)^{\alpha}}+ h^{\alpha}\Gamma\left(2-\alpha\right)\rho\left(h\right).
\end{equation}
Since the sum above is telescoping, we have
\begin{equation}
    \begin{aligned}
        \left\|u^1\right\|_{1, h} \leq b_n - b_{n+1} +\left(1 + \left(n-1\right)h^{\alpha}\Gamma\left(2-\alpha\right)\rho\left(h\right)\right)\left(1-b_{n}\right)+ \frac{1-\alpha}{\left(n+1\right)^{\alpha}}+ h^{\alpha}\Gamma\left(2-\alpha\right)\rho\left(h\right) = \\ =  1 -  b_{n-1} + \frac{1-\alpha}{\left(n+1\right)^{\alpha}} + \left(n-b_n\right)h^{\alpha}\Gamma\left(2-\alpha\right)\rho\left(h\right) \leq 1 -  b_{n-1} + \frac{1-\alpha}{\left(n+1\right)^{\alpha}} + nh^{\alpha}\Gamma\left(2-\alpha\right)\rho\left(h\right).
    \end{aligned}
\end{equation}
Using Taylor expansion, we obtain
\begin{equation}
    b_{n+1}\geq \frac{1-\alpha}{n^{\alpha}} + \frac{\alpha\left(1-\alpha\right)}{2}\frac{1}{n^{\alpha+1}}.
\end{equation}
Thus, we have
\begin{equation}
    \left\|u^1\right\|_{1, h} \leq 1  - \left(1-\alpha\right)\left(\frac{1}{n^{\alpha}} - \frac{1}{\left(n+1\right)^{\alpha+1}}\right) - \frac{\alpha\left(1-\alpha\right)}{2}\frac{1}{n^{\alpha+1}} + nh^{\alpha}\Gamma\left(2-\alpha\right)\rho\left(h\right) \leq 1  + nh^{\alpha}\Gamma\left(2-\alpha\right)\rho\left(h\right),
\end{equation}
what ends the proof.
\end{proof}
\end{theorem}

\section{Numerical simulations}
In this section, we illustrate our results with several numerical examples. First, the discretization of the Dirac delta is performed using the cosine approximation in the form \cite{di2021remarks}
\begin{equation}
    \delta\left(x\right) \approx \frac{1 + \cos\left(\frac{\pi x}{Kh}\right)}{2Kh}\mathbf{1}_{\left\{-Kh\leq x \leq Kh\right\}}\left(x\right),
\end{equation}
where $h$ is the discretization step, $K\in \mathbb{N}_0$ is the chosen number (we use $K=2$) and $\mathbf{1}$ is an indicator. In Figures \ref{fig::wfp}, \ref{fig::jfp} and \ref{fig::cp} we compare numerical and analytical solutions (using \ref{combinationscheme2}) to the equations corresponding to the three types of L\'evy walks for $p\in \left\{0.05, 0.25, 0.5\right\}$ and $\alpha=0.5$. We set $h=2^{-11}$ and $T=1$. As can be seen, the validity of the scheme is confirmed with good accuracy. 

\begin{figure}
	\centering
	\includegraphics[width=1\textwidth, height=0.2\textheight]
	{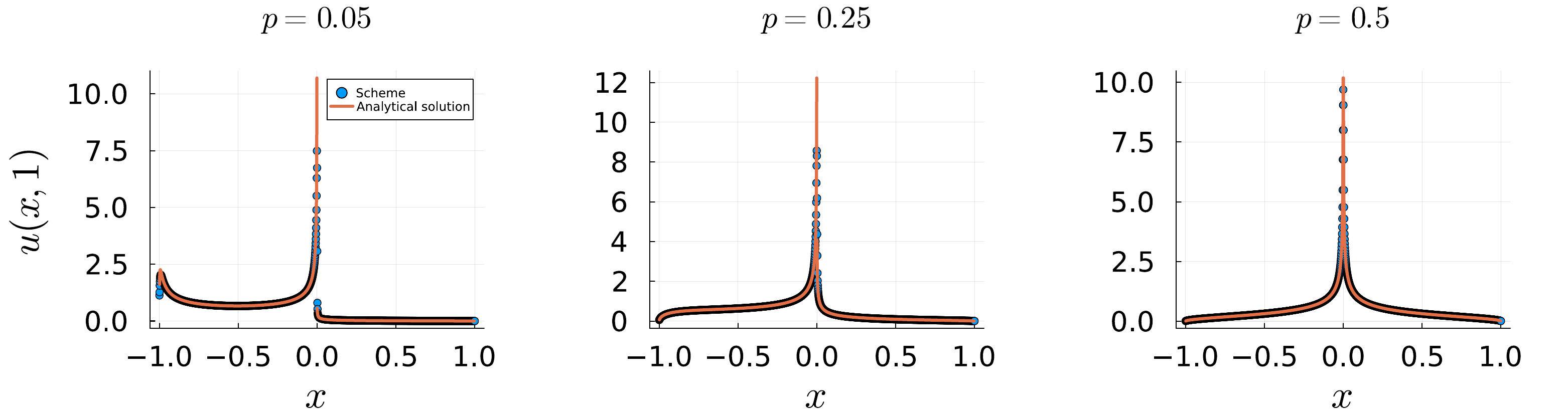}
	\caption{Comparison of analytical solution and numerical scheme for WF L\'evy walk for $\alpha=0.5$ and different values of $p$.
	}
    \label{fig::wfp}
\end{figure}

\begin{figure}
	\centering
	\includegraphics[width=1\textwidth, height=0.2\textheight]
	{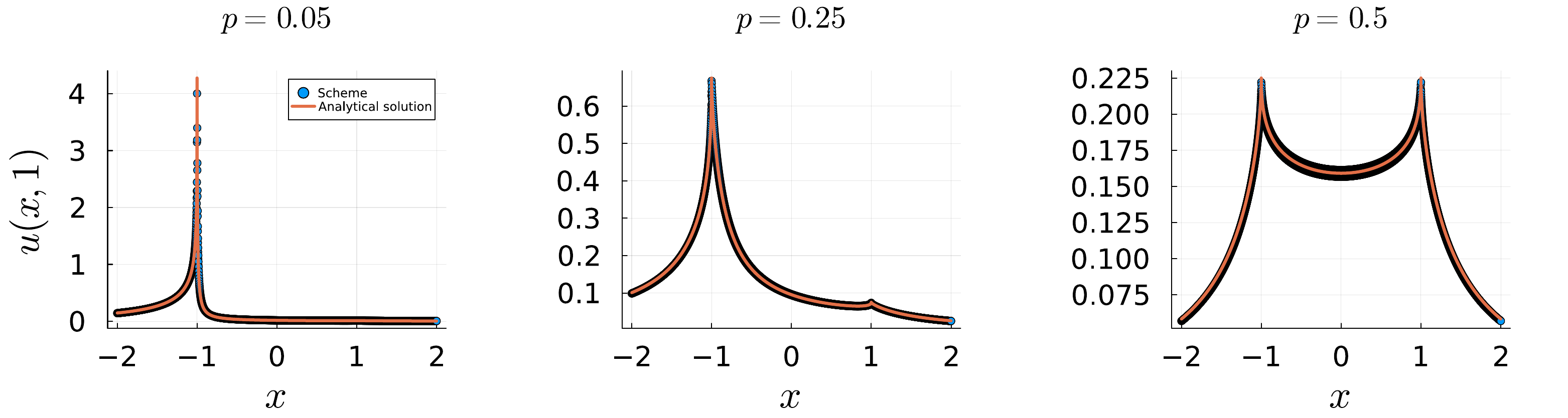}
	\caption{Comparison of analytical solution and numerical scheme for JF L\'evy walk for $\alpha=0.5$ and different values of $p$.
	}
    \label{fig::jfp}
\end{figure}

\begin{figure}
	\centering
	\includegraphics[width=1\textwidth, height=0.2\textheight]
	{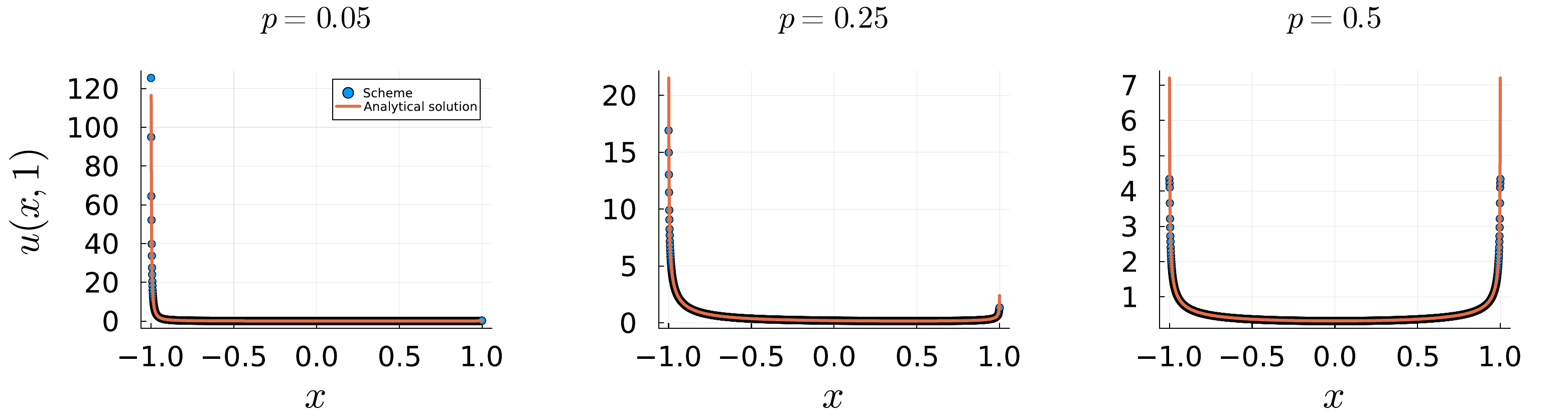}
	\caption{Comparison of analytical solution and numerical scheme for standard L\'evy walk for $\alpha=0.5$ and different values of $p$.
	}
    \label{fig::cp}
\end{figure}

Similarly, in Figures \ref{fig::wfa}, \ref{fig::jfa} and \ref{fig::ca} we compare numerical and analytical solutions to equations corresponding to the three types of L\'evy walks for $\alpha\in \left\{0.25, 0.5, 0.75\right\}$ and $p=0.25$. We set $h=2^{-11}$ and $T=1$. In all cases, we observe a very good accuracy of our scheme.

\begin{figure}
	\centering
	\includegraphics[width=1\textwidth, height=0.2\textheight]
	{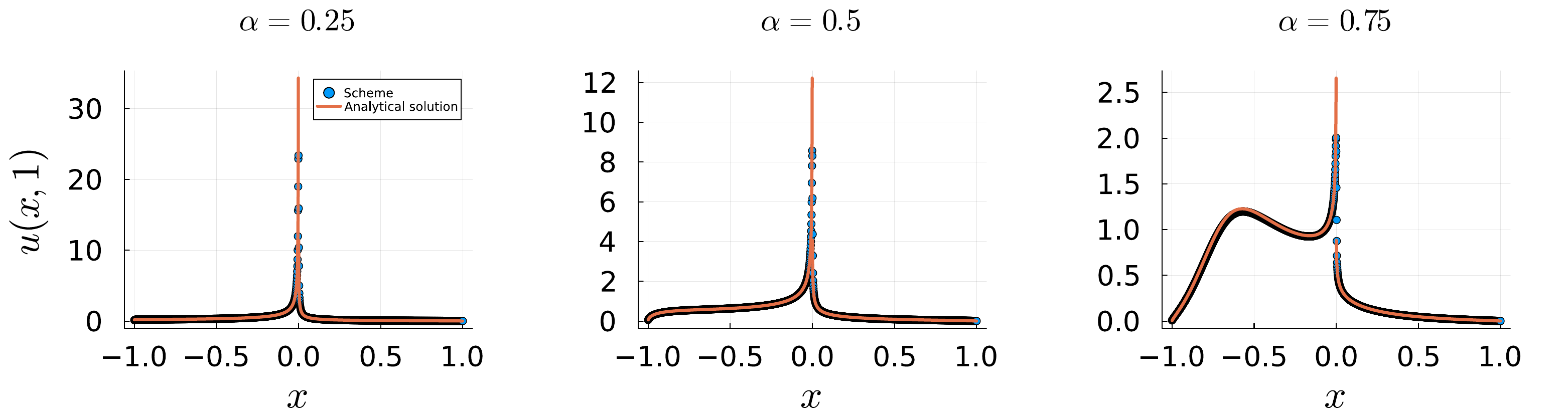}
	\caption{Comparison of analytical solution and numerical scheme for WF L\'evy walk for $p=0.25$ and different values of $\alpha$.
	}
    \label{fig::wfa}
\end{figure}

\begin{figure}
	\centering
	\includegraphics[width=1\textwidth, height=0.2\textheight]
	{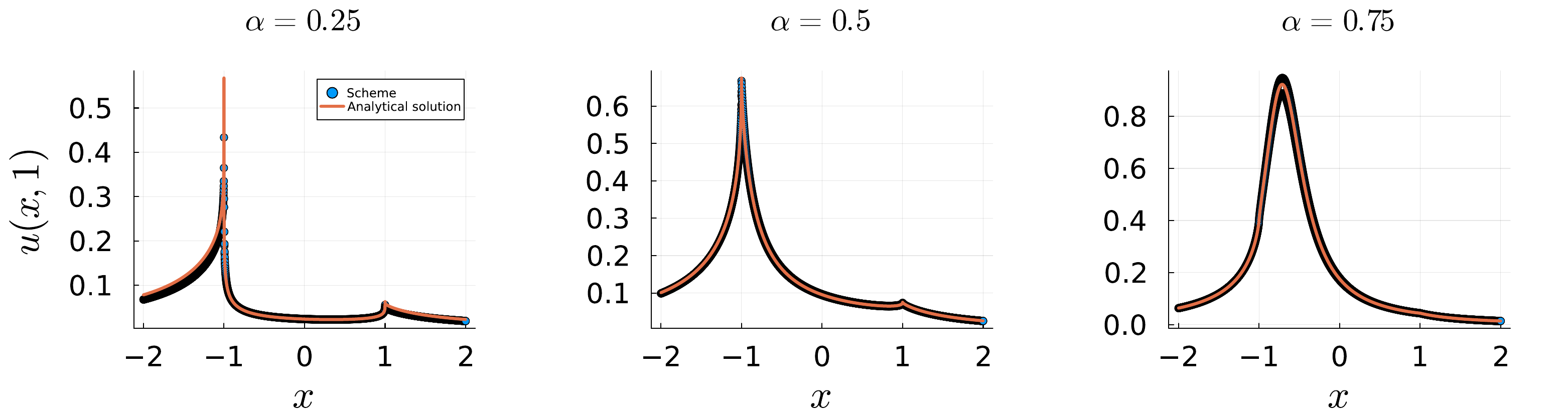}
	\caption{Comparison of analytical solution and numerical scheme for JF L\'evy walk for $p=0.25$ and different values of $\alpha$.
	}
    \label{fig::jfa}
\end{figure}

\begin{figure}
	\centering
	\includegraphics[width=1\textwidth, height=0.2\textheight]
	{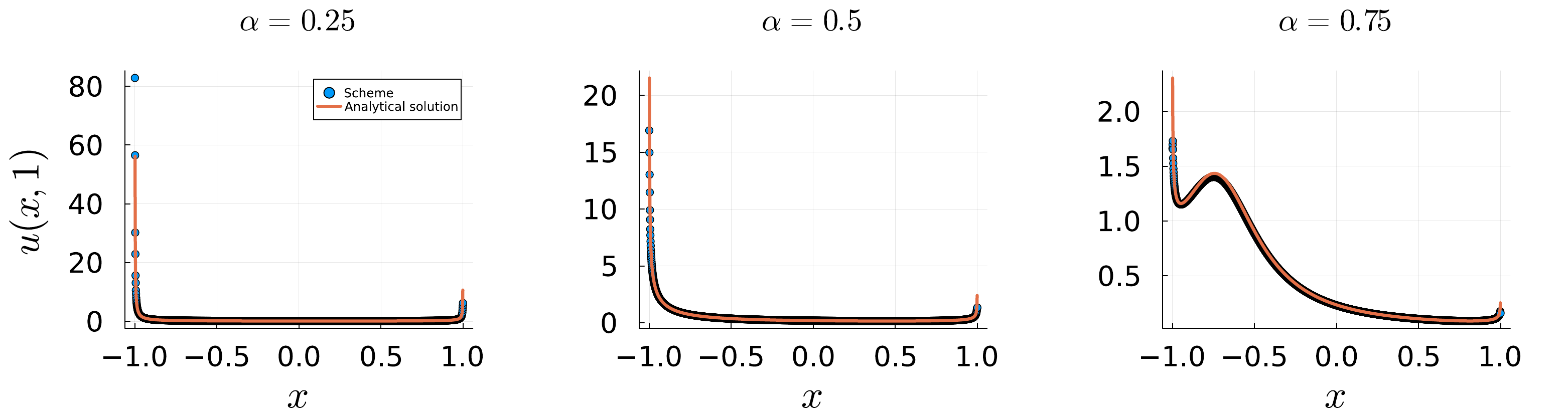}
	\caption{Comparison of analytical solution and numerical scheme for standard L\'evy walk for $p=0.25$ and different values of $\alpha$.
	}
    \label{fig::ca}
\end{figure}

In Figures \ref{fig::probcons1} and \ref{fig::probcons2}, we show how well each scheme preserves the total probability over time for the equation corresponding to the scaling limit of the wait-first L\'evy walk. Both the conventional finite-volume upwind method (\eqref{combinationscheme}, Fig. \ref{fig::probcons1}) and our modified version, which evaluates the source term using the next-time level value (\eqref{combinationscheme2}, Fig. \ref{fig::probcons2}) are conservative, that is, the total probability approaches over time, from either side, the value $1$. Note that immediately after $t=0$, the initial condition of the Dirac delta $u(x,0)=\delta(x)$ distorts the initial probability values. Beyond that initial transient, the two methods exhibit opposite behaviors: the standard upwind discretization yields a steadily decreasing total probability, whereas the modified scheme produces a monotonically increasing curve. In particular, the monotonic growth of the total probability in the modified scheme is rigorously ensured by Proposition \ref{probconsprop}.

\begin{figure}
	\centering
	\includegraphics[width=0.6\textwidth, height=0.3\textheight]
	{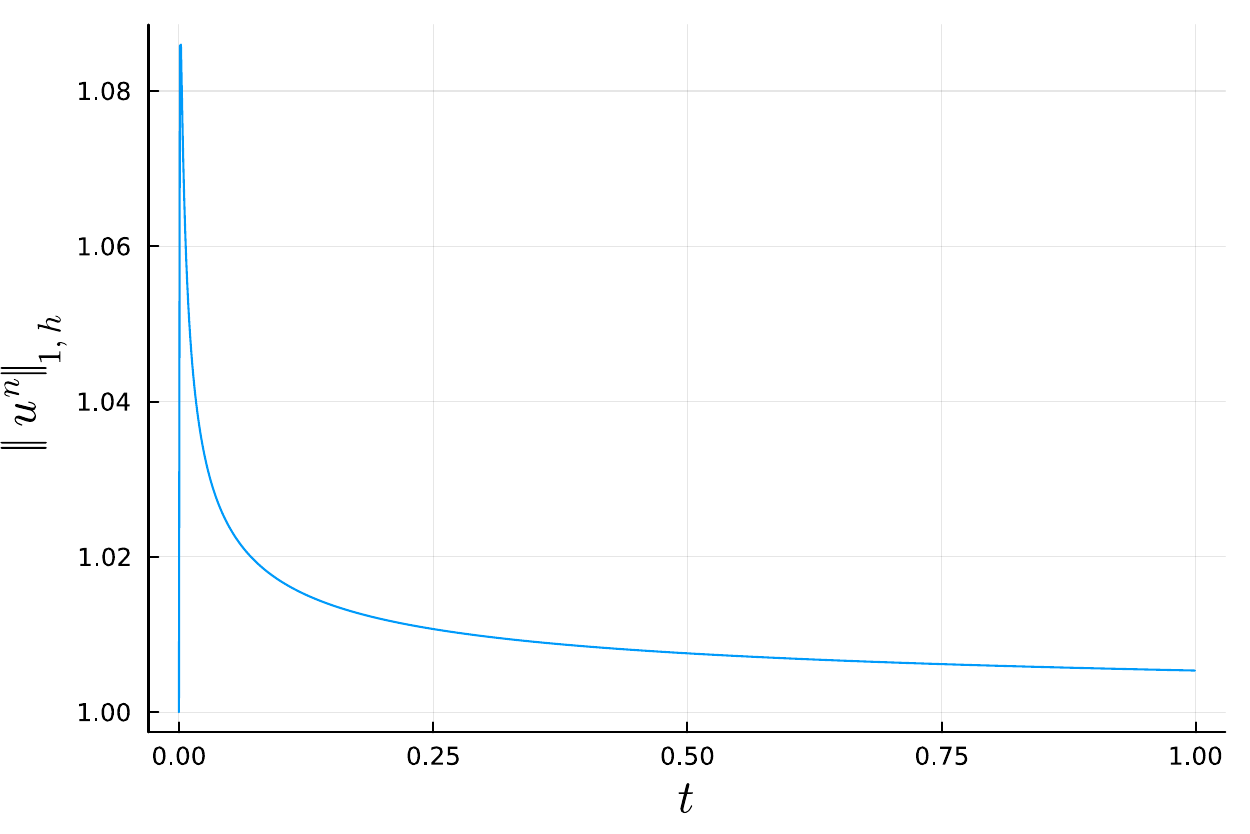}
	\caption{Visualization of the conservation law for the total probability for the consecutive time points of the numerical solution to a PDF problem of the scaling
		limit of the wait-first Lévy walk (Example \ref{wfex}), using standard finite volume upwind scheme \eqref{combinationscheme}. The parameters used in the numerical scheme are $\alpha=0.5, \, p=0.5, \, h=2^{-10}$.}
    \label{fig::probcons1}
\end{figure}

\begin{figure}
	\centering
	\includegraphics[width=0.6\textwidth, height=0.3\textheight]
	{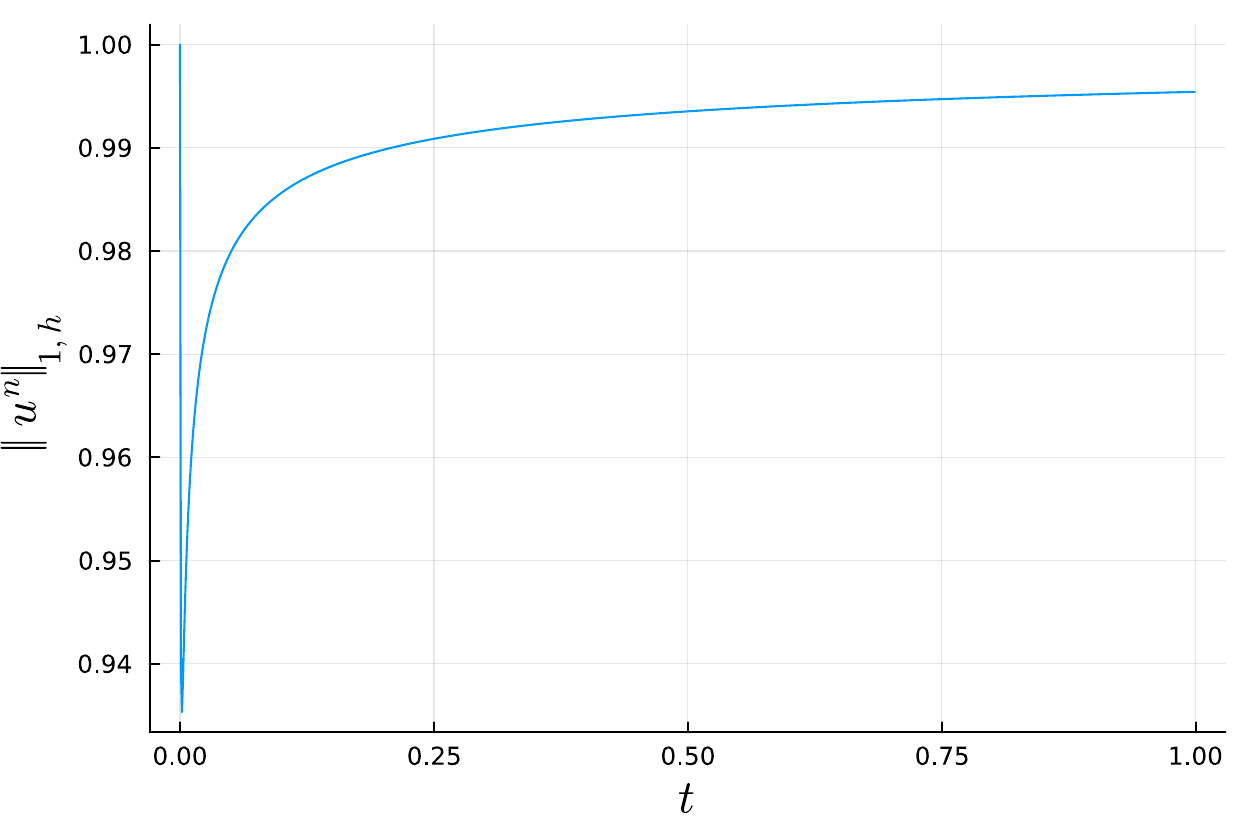}
	\caption{Visualization of the conservation law for the total probability for the consecutive time points of the numerical solution to a PDF problem of the scaling
		limit of the wait-first Lévy walk (Example \ref{wfex}), using modified finite volume upwind scheme \eqref{combinationscheme2}. The parameters used in the numerical scheme are $\alpha=0.5, \, p=0.5, \, h=2^{-10}$.}
    \label{fig::probcons2}
\end{figure}

We now test our scheme on the homogeneous problem \eqref{simplecombinationeq} with $u\left(x,0\right)=0$ and $f\left(x,t\right) = t^{\mu}$ with $\mu>0$. In this case, the solution has analytical form $u\left(x,t\right) = \Gamma\left(\mu+1\right)t^{\mu+\alpha}/\Gamma\left(\mu+\alpha+1\right)$. This closed-form expression allows us to measure the numerical error directly and extract the empirical convergence order. Figures \ref{fig::etlog} and \ref{fig::et2log} display these convergence rates for $\mu=1,2$, mesh sizes $h\in\left\{2^{-12}, 2^{-11}, \ldots, 2^{-4}\right\}$ and $\alpha\in\left\{0.1,0.25,0.5,0.75,0.9\right\}$. Our numerical experiments show that for $\mu=1,2$ the convergence rate of the scheme under the $L^{\infty}$ norm is $r=2-\alpha$ which confirms the results of the standard literature for the case of high regularity \cite{plociniczak2023linear}. On the other hand, in Figure \ref{fig::el2} we show the rate of convergence for the PDF of the scaling limits of the first wait-wait L'evy walk for mesh sizes $h\in\left\{2^{-12}, 2^{-11}, \ldots, 2^{-4}\right\}$ and $\alpha\in\left\{0.1,0.25,0.5,0.75,0.9\right\}$. For all values of $\alpha$, the estimated convergence rate in the norm $L^2$ is equal to 1, which confirms the theoretical results of Theorem \ref{convthm} and is consistent with the discretization L1 of the Caputo derivative calculated away from the initial time (see \cite{stynes2017error}). 

\begin{figure}
	\centering
	\includegraphics[width=0.6\textwidth, height=0.3\textheight]
	{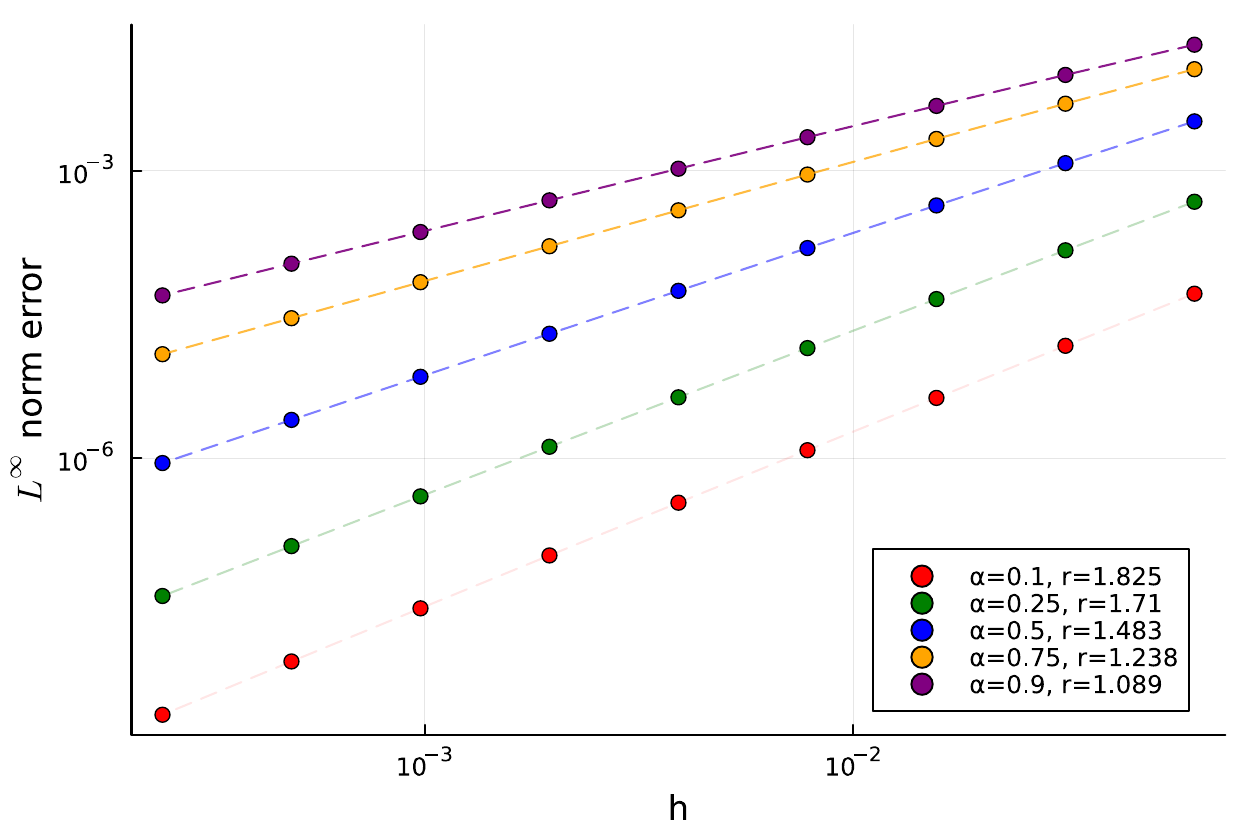}
	\caption{The $L^{\infty}$ norm error for $h\in \left\{2^{-12}, 2^{-11}, \ldots, 2^{-4}\right\}$ and $\alpha \in \left\{0.1, 0.25, 0.5, 0.75, 0.9\right\}$ for the numerical solution to \eqref{simplecombinationeq} with $u\left(x, 0\right) = 0$, $f\left(x,t\right) = t$ and $p=0.5$ in log-log scale.
	}
    \label{fig::etlog}
\end{figure}

\begin{figure}
	\centering
	\includegraphics[width=0.6\textwidth, height=0.3\textheight]
	{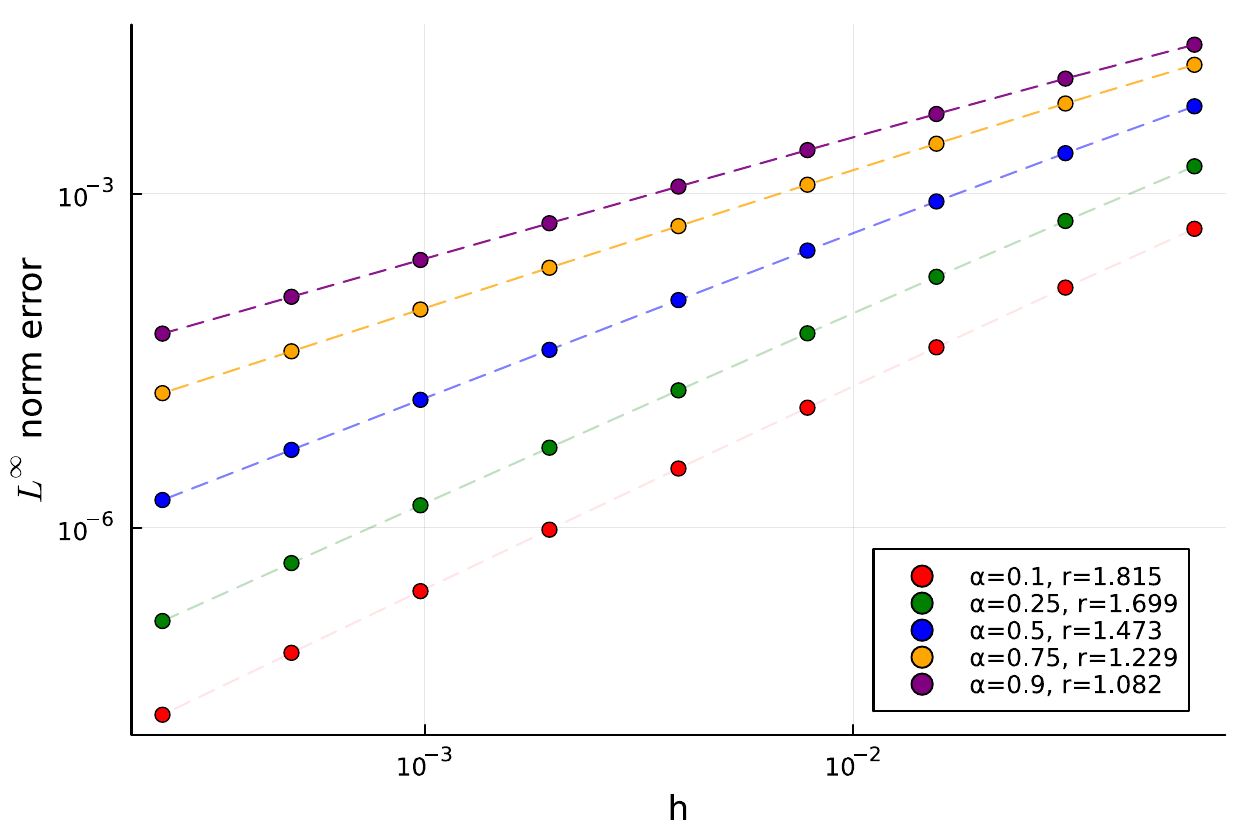}
	\caption{The $L^{\infty}$ norm error for $h\in \left\{2^{-12}, 2^{-1}, \ldots, 2^{-4}\right\}$ and $\alpha \in \left\{0.1, 0.25, 0.5, 0.75, 0.9\right\}$ for the numerical solution to \eqref{simplecombinationeq} with $u\left(x, 0\right) = 0$, $f\left(x,t\right) = t^2$ and $p=0.5$ in log-log scale.
	}
    \label{fig::et2log}
\end{figure}

\begin{figure}
	\centering
	\includegraphics[width=0.6\textwidth, height=0.3\textheight]
	{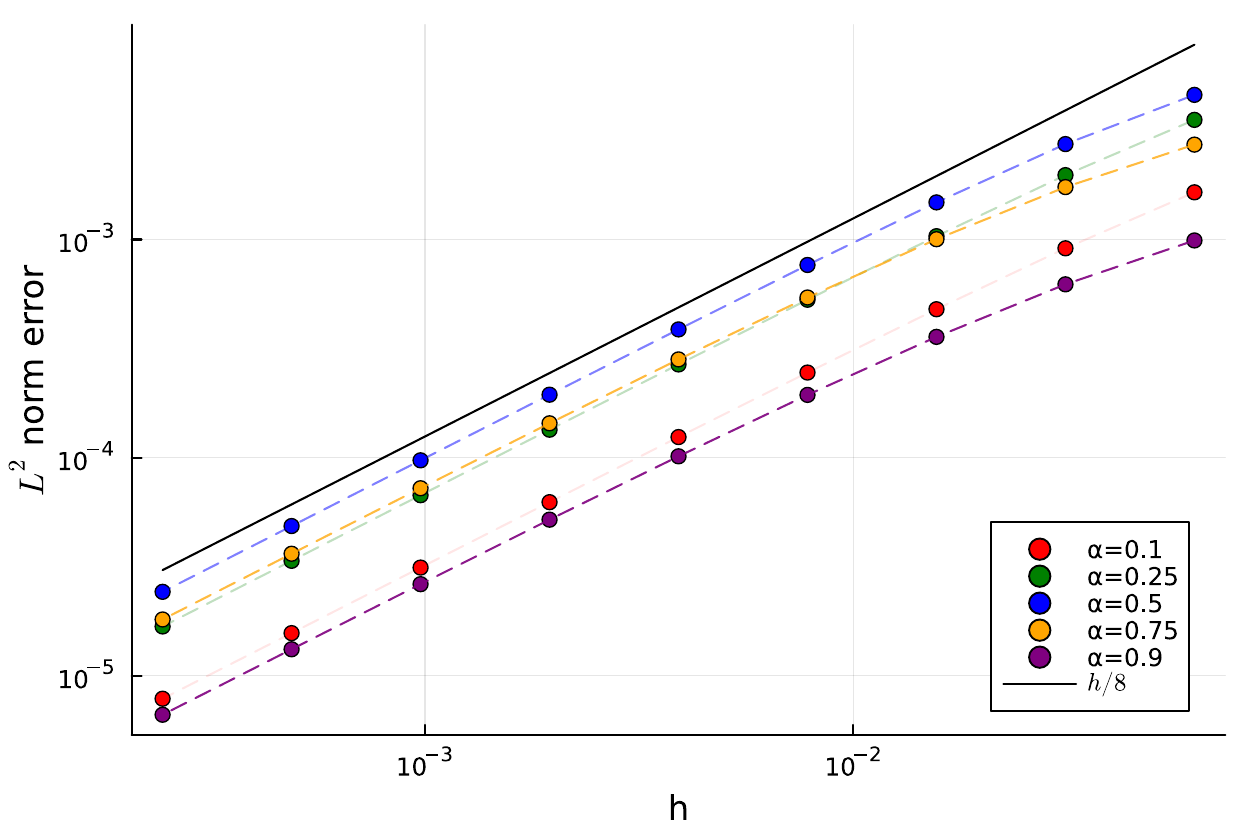}
	\caption{The $L^{2}$ norm error for $h\in \left\{2^{-12}, 2^{-1}, \ldots, 2^{-4}\right\}$ and $\alpha \in \left\{0.1, 0.25, 0.5, 0.75, 0.9\right\}$ for the numerical solution of the PDF problem of the scaling limit of the  wait-first L\'evy walk (see Example \ref{wfex}) in log-log scale.
	}
    \label{fig::el2}
\end{figure}

\section{Conclusions}
In this work, we have introduced and analyzed a class of linear transport equations driven by a convex combination of fractional material derivatives, motivated by the deterministic limits of L'evy walks. This is an important generalization of the results presented in \cite{Plociniczak2024}, where only equations driven solely by a single fractional material derivative were considered. We proved the existence and uniqueness of (mild) solutions. Using a recently developed pointwise representation of the fractional material derivative, we identified a sharp necessary and sufficient condition that guarantees that the solution remains a probability density for all times (nonnegativity and preservation of total mass). To enable reliable computations that respect these probabilistic properties, we designed a finite-volume discretization that is mass conservative by construction. We established discrete stability and convergence to the continuous solution as the mesh is refined, and demonstrated through numerical experiments that the scheme preserves positivity, conserves mass with respect to discretization error, and reproduces the anomalous spreading phenomena associated with asymmetric L\'evy-walk dynamics, i.e. the scaling limits of L\'evy-walks with the probabilities of jumps to the right and to the left not being equal to each other. In general, the combined theoretical and numerical results provide a consistent and practical framework for deterministic modeling and simulation of anomalous transport through fractional material dynamics.

Promising extensions of this study include nonlinear models, variable coefficients and heterogeneous media,  higher dimensions and anisotropy, and sharper error estimates and convergence rates (relevant in the presence of singularity). 

\section*{Acknowledgments}
This work has been supported by the National Science Centre, Poland (NCN) under the grant Sonata Bis with a number NCN 2020/38/E/ST1/00153. 

\bibliographystyle{plain}
\bibliography{bibliography.bib}
\appendix
\end{document}